\documentclass{amsart}

\usepackage{amssymb, amsmath}
\usepackage{color}
\usepackage{array}

\usepackage[pdftex]{hyperref}
\hypersetup{
    colorlinks=true, 
    linktoc=all,    
    linkcolor=blue,  
}

\newcommand{\bb}{\mathbb}
\newcommand{\cal}{\mathcal}

\newcommand{\lb}{\left\{}
\newcommand{\rb}{\right\}}
\newcommand{\la}{\left<}
\newcommand{\ra}{\right>}

\newcommand{\gch}{\mathsf{GCH}}

\newcommand{\pfa}{\mathsf{PFA}}

\newcommand{\ps}{\mathbb{P}}
\newcommand{\Q}{\mathbb{Q}}
\newcommand{\R}{\mathbb{R}}

\newcommand{\al}{\alpha}
\newcommand{\be}{\beta}
\newcommand{\ga}{\gamma}

\newcommand{\de}{\delta}
\newcommand{\De}{\Delta}

\newcommand{\ka}{\kappa}
\newcommand{\lam}{\lambda}

\newcommand{\om}{\omega}

\newcommand{\bsl}{\setminus}

\newcommand{\seq}{\subseteq}
\newcommand{\we}{\wedge}

\newcommand{\dom}{\operatorname{dom}}

\newcommand{\es}{\emptyset}

\theoremstyle{definition}
\newtheorem{definition}{Definition}[section]

\newtheorem{question}[definition]{Question}
\newtheorem{notation}[definition]{Notation}

\theoremstyle{plain}

\newtheorem{theorem}[definition]{Theorem}
\newtheorem{proposition}[definition]{Proposition}
\newtheorem{lemma}[definition]{Lemma}

\newtheorem{remark}[definition]{Remark}

\newtheorem{assumption}[definition]{Assumption}

\begin{document}

\title[Forcing Axioms for Preserving a Topological Property]{Forcing Axioms for Proper Posets Preserving a Topological Property: Consistency Results}

\author{Thomas Gilton}\address{University of Pittsburgh
Department of Mathematics. The Dietrich School of 
Arts and Sciences, 301 Thackeray Hall,
Pittsburgh, PA, 15260 United States
} \email{tdg25@pitt.edu} \urladdr{https://sites.pitt.edu/~tdg25/}

\date{\today}

\begin{abstract}
   Forcing axioms are generalizations of Baire category principles that allow one to intersect more dense open sets and to do so in a wider variety of circumstances. In this paper we introduce two new forcing axioms related to posets which preserve topological properties of various spaces, specifically the properties of Lindel{\"o}f and countably tight. The focus in this paper is on using Neeman's side conditions iteration schema (\cite{Neemanttm}) to prove the consistency of these two forcing axioms. In later work, we will discuss applications of these forcing axioms in more detail.
\end{abstract}

\maketitle

\tableofcontents

\section{Introduction}

In a typical Baire category argument, one constructs an object of interest (say, an everywhere continuous, nowhere differentiable function) by refining approximations (dense open sets) and using features of the topology (say, being a compact metric space) to ensure that the dense sets ``hone in" on the desired object (i.e., that any countably-many have a non-empty intersection). This is in contrast to constructions where one writes down an explicit formula for the object under consideration (say, the formula for the sawtooth function).

Forcing axioms generalize Baire category arguments by allowing the intersection of more (at least $\aleph_1$-many) dense open sets and under a wider variety of circumstances, namely, for posets with various desirable features. Martin's Axiom was the first of these. $\mathsf{MA}(\ka)$ asserts that if $\ps$ is a poset with the countable chain condition (c.c.c.) and if $\la D_\al:\al<\ka\ra$ is a sequence of dense subsets of $\ps$, then there is a filter $G\subset\ps$ so that for all $\al<\ka$, $G\cap D_\al\neq\es$. As another example, the Proper Forcing Axiom, abbreviated $\pfa$, is a substantial strengthening of Martin's Axiom. It asserts that if $\ps$ is a proper poset (we will review this definition later) and if $\la D_\al:\al<\om_1\ra$ is a sequence of $\om_1$-many dense subsets of $\ps$, then there is a filter $G\subset\ps$ so that for all $\al<\om_1$, $D_\al\cap G\neq\es$. 

Models satisfying forcing axioms are normally obtained by reaching a closure point in a reasonably well-behaved iteration of forcings. Accordingly, constructing such a model involves at least the following two ingredients:
\begin{enumerate}
    \item the appropriate notion of a ``closure point" and
    \item an iteration theorem, showing that a desirable property of posets is preserved under a specific kind of iteration.
\end{enumerate}
For example, with Martin's Axiom, one only has to deal with c.c.c. posets of cardinality below the targeted value of the continuum. Thus any regular cardinal $\ka\geq\om_2$ can act as a closure point. Moreover, a finite-support iteration of c.c.c. posets is still c.c.c. -- the classic result of Solovay and Tennenbaum \cite{SolovayTennenbaum} --  giving us the desired iteration theorem. By contrast, in the case of $\pfa$, one needs a large cardinal (all known proofs to date involve a supercompact) in order to reach the right closure point. The relevant iteration theorem says that the property of being proper is preserved under countable support iterations (\cite{ShelahPIPF}).

A recent development is Neeman's proof (\cite{Neemanttm}) of the consistency of $\pfa$ using a forcing that has \emph{finite} supports. This gives additional flexibility in constructing models of forcing axioms. A condition in Neeman's forcing includes two components. The first component, known as the \emph{working part}, is a finite support function analogous to a condition in a traditional iteration. The second component is a finite, $\in$-increasing sequence of models called \emph{side conditions}. The two components are related by requiring the outputs of the working part to be (forced to be) generic conditions for plenty of the models in the side condition. This interaction between the two parts of a condition allows Neeman to argue that the entire poset is proper, and hence preserves $\omega_1$. For other examples of such ``side conditions iteration theorems" see \cite{AsperoMota} and \cite{GiltonNeeman}.

The idea of using side conditions in order to help argue for properness goes back to Todor{\u c}evi{\' c} (see  \cite{TodorcevicSideConditions} and later expanded to include preservation of $\om_1$ and $\om_2$ in \cite{TodorcevicDirected}). The idea is remarkably fecund: by introducing models themselves into the conditions in a forcing, one can place natural constraints on how the models interact with the other components of a condition and thereby more smoothly argue for properness. Some other highlights include the following: \cite{Koszmider}, \cite{Friedman}, \cite{MitchellOuch}, \cite{GiltonKruegerFun}, and \cite{KruegerMota}, among many others.\\

In this work, we introduce two forcing axioms for proper posets preserving some topological property $\Phi$ of a given space $X$, specifically,  Lindel{\"o}f spaces as well as countably tight spaces. The models will be constructed using  Neeman's finite-support method. We view the situation as analogous to the results in \cite{TodorcevicSouslinTree} and \cite{KruegerForcingAxiom}, where those authors focus on proper posets preserving some noteworty feature of a tree on $\om_1$.

More specifically, we define forcing axioms of the form $\pfa_\Phi(X)$, where $X$ is a topological space satisfying the topological property $\Phi$. Such an axiom asserts that if $\ps$ is a proper poset that preserves $\Phi$ of $X$ and if $\la D_\al:\al<\om_1\ra$ is a sequence of dense subsets of $\ps$, then we can find a filter meeting each $D_\al$. In the cases where $\Phi$ equals (a) countably tight or (b) Lindel{\"o}f, and where $X$ is a space of size less than a given supercompact cardinal, we provide proofs of the consistency of these axioms: see Theorems \ref{theorem:Lindelof} and \ref{theorem:CT}.  Note that in certain trivial cases (say, when $X$ is at most countably infinite), the axiom $\pfa_\Phi(X)$ is the same as $\pfa$. However, the question of the consequences of $\pfa_\Phi(X)$ for non-trivial $X$ is much more interesting. We will discuss the ``spectrum of consequences" question in further work.

The preservation of topological properties under various forcings has been an active area of research. A significant discovery is the fact that an $\om_1$-closed forcing needn't preserve the Lindel{\"o}f property in general (\cite{ScheepersCountableTightness}). However, Dow showed that if one first adds $\om_1$-many Cohen reals, and then forces with a poset which is $\om_1$-closed \emph{in the extension}, then such a two-step poset preserves the property of being countably tight (\cite{DowTwoApplications}) and the Lindel{\"o}f property (\cite{DowSubmodels}). Iwasa (\cite{Iwasa}) later studied the preservation of covering properties under Cohen forcing, and Kada (\cite{Kada})  studied the preservation of covering properties under a more general type of forcing (so-called ``weakly endowed").

The present author, together with Holshouser (\cite{GiltonHolshouser}), generalized Dow's results by showing that strongly proper forcings\footnote{We worked with posets that are strongly proper for ``stationarily-many" countable elementary submodels; this is a strictly weaker notion than being strongly proper for a club of countable elementary submodels. The latter is typically taken to be the definition of a ``strongly proper" poset.} preserve a number of covering properties, including strategic versions, as well as countable tightness. Marun (\cite{Marun}) showed in his PhD thesis -- both earlier and independently of \cite{GiltonHolshouser} -- that posets which are strongly proper for enough countable elementary submodels preserve the Lindel{\"o}f property. Some other noteworthy examples in this line of research include the following: \cite{DowRemote}, \cite{DTWNewProofs}, \cite{GJTForcingAndNormality}, and \cite{Kada2}.

The present work is organized in the following way: Section \ref{Section:background} briskly reviews the background on proper and strongly proper forcings, as well as the relevant topological definitions. We also define the forcing axiom $\pfa_\Phi(X)$. Section \ref{Section:NeemanSequence} reviews the basics of Neeman's sequence poset, highlighting various lemmas that we will refer to throughout the course of the present work. This section is included in order to make the current paper self-contained; however the reader is encouraged to read Neeman's paper for their full (and necessary) dosage. In Section \ref{Section:Poset} we define the type of poset that will provide models of the forcing axioms discussed in this paper. In Section \ref{Section:EmbeddingLemmas}, we prove embedding lemmas which will be used in our later inductive arguments, and in Section \ref{Section:SCAPreservation} we prove preservation lemmas that will be used in conjunction with the results of Section \ref{Section:EmbeddingLemmas}. In Sections \ref{Section:PFAL} and \ref{Section:PFACT} we discuss the other cases in the inductive proofs of Theorems \ref{theorem:Lindelof} and \ref{theorem:CT}, and then we stitch everything together to wrap up the proofs. Section \ref{Section:questions} contains open questions and questions for further research.

\section{Background Definitions}\label{Section:background}

In this section, we review the basic definitions from forcing and from topology. For the former, we will review the definitions associated with proper and strongly proper posets, as well as some of their foundational features. For the latter, we will review the definitions of the relevant topological properties as well as review the definition of what it means for a forcing to preserve a topological property.

We begin with the notion of \emph{properness} for a poset which has its roots in Shelah's work (\cite{ShelahIndependence},\cite{ShelahPIPF}).

\begin{definition}\label{def:proper}
    Let $\kappa$ be a regular uncountable cardinal, $\ps\in H(\kappa)$ a poset, and $M\prec H(\kappa)$ with $\ps\in M$.

    A condition $q\in \ps$ is an $(M,\ps)$-\emph{generic} condition if it forces that
    $$
    M\cap\dot{G}\cap D\neq\es
    $$
    for each $D\in M$ which is dense in $\ps$.

    $\ps$ is \emph{proper for $M$} if for every $p\in M\cap\ps$, there is an extension $q\leq p$ so that $q$ is $(M,\ps)$-generic.

    Finally, $\ps$ is simply said to be \emph{proper} if for some (equivalently, all) large enough uncountable regular $\lambda$, there is a club $\mathcal{C}$ of countable elementary submodels of $H(\lambda)$ so that $\ps$ is proper for each $M\in\mathcal{C}$.
\end{definition}

The following lemma follows from applications of elementarity and the definition of properness. We will use it frequently throughout the paper:

\begin{lemma}\label{lemma:genericforincreasingsequence} Let $\kappa$ be a regular uncountable cardinal. Suppose that $M_0\in\dots\in M_{n-1}$ is an $\in$-increasing sequence of countable elementary submodels of $H(\kappa)$ and that $\ps\in H(\kappa)$ is a poset which is proper for each $M_i$. Next, let $p\in\ps$, and suppose that for some $k<n-1$, $p$ is $(M_i,\ps)$-generic for all $i\leq k$ and that $p\in M_{k+1}$. Then there exists $p'\leq p$ so that $p'$ is an $(M_i,\ps)$-generic condition for all $i<n$.
\end{lemma}

The following are some useful equivalences to a condition being $(M,\ps)$-generic; proofs may be found in \cite{Jech}.

\begin{lemma}
    Let $\kappa$, $\ps$, and $M$ be as in Definition \ref{def:proper}. Let $q\in\ps$. The following are equivalent:
    \begin{enumerate}
        \item $q$ is $(M,\ps)$-generic;
        \item $q\Vdash M[\dot{G}]\cap V=M$;
        \item for all $q^*\leq q$ and all dense $D\seq\ps$ with $D\in M$, there exists a $p\in M\cap D$ so that $p$ is compatible with $q^*$.
    \end{enumerate}
\end{lemma}

Note that for a condition $q$ to be $(M,\ps)$-generic, we only have to check a certain property with respect to dense subsets of $\ps$ which are members of $M$. A significant strengthening of this idea was first made explicit in \cite{MitchellOuch}, namely, the idea of a \emph{strongly} generic condition.

\begin{definition}
    Let $\kappa$, $\ps$, and $M$ be as in Definition \ref{def:proper}. A condition $q\in\ps$ is said to be an $(M,\ps)$-\emph{strongly generic} condition if $q$ forces that $\dot{G}\cap M$ is a $V$-generic filter for $\ps\cap M$.
\end{definition}

Note that for $q$ to be $(M,\ps)$-strongly generic, $q$ must force that $D'\cap\dot{G}\cap M\neq\es$ for \emph{all} dense $D'\seq M\cap\ps$ with $D'\in V$, not just ones which are the trace of some element of $M$.

In this paper, we shall be concerned with only the following two topological properties:

\begin{definition}\hfill
    \begin{itemize}
        \item A space $(X,\tau)$ is said to be
        \emph{Lindel{\"o}f} if every open cover of $X$ has a countable subcover.
        \item A space $(X,\tau)$ is said to be \emph{countably tight} if for each $A\seq X$ and $x\in X$ with $x\in\operatorname{cl}(A)$, there is a countable $B\seq A$ so that $x\in\operatorname{cl}(B)$.
    \end{itemize}
\end{definition}

We next review the definition of preserving a topological property by a forcing:

\begin{definition}
    Let $(X,\tau)$ be a topological space and $\ps$ a poset. Fix a $V$-generic filter $G$ over $\ps$. In $V[G]$, noting that $\tau$ forms a basis for a topology on $X$, let $\tau_G$ denote the topology on $X$ generated by $\tau$.

    Let $\Phi$ be a topological property which $X$ satisfies in $V$. We say that $\ps$ \emph{preserves that $X$ satisfies $\Phi$} if $\ps$ forces that $(X,\tau_{\dot{G}})$ satisfies $\Phi$.
\end{definition}

We give a few toy examples:
\begin{enumerate}
    \item Let $x\in X$ be a point of first countability in an arbitrary space $X$. Every forcing preserves that $x$ is a point of first countability, as a countable neighborhood base at $x$ in $V$ remains such in any generic extension (since the ground model topology is used as a base in the generic extension).
    \item Let $\ps$ be a forcing that adds a new real. Let $Z$ be the unit interval $[0,1]^V$ in the reals, \emph{as defined in $V$}. Naturally, $Z$ is a compact subset of $\R^V$ in $V$. However, in any extension by $\ps$, $Z$ is no longer a compact subset of $\R^V$. Indeed, let $r$ be a new real in $V[G]$, where we assume $0<r<1$. Then the following is an open cover of $[0,1]^V$ in the space $\R^V$ in the model $V[G]$:
    $$
    \lb [0,q)^V:q\in\Q \wedge 0<q<r\rb\cup\lb (s,1]^V:s\in\Q\wedge r<s<1\rb.
    $$
    However, this open cover has no finite subcover. We emphasize that we are working with rational intervals from $\R^V$, but we are using the new real $r$ to define the open cover.
\end{enumerate}

Now we define our forcing axiom schema:

\begin{definition}\label{def:axioms}
    Let $X$ be a topological space satisfying a topological property $\Phi$. The axiom $\pfa_\Phi(X)$ asserts that for any proper poset $\Q$ so that $\Q$ preserves that $X$ has $\Phi$ and for any collection $\lb D_\al:\al<\om_1\rb$ of dense subsets of $\Q$, there is a filter $H$ so that $H$ meets each $D_\al$.
\end{definition}

\section{Neeman's Model Sequence Poset}\label{Section:NeemanSequence}

In this section, we will review some of the main results from Section 2 of \cite{Neemanttm} in order to make the current paper reasonably self-contained. We will not be providing proofs of most of these results, and the interested reader is heartily encouraged to carefully read at least Section 2 of \cite{Neemanttm}.

We begin with the basic set-up of the model sequence poset.

\begin{notation}
    Let $\theta$ be a large enough regular cardinal. Let $\cal{S}$ denote a set of countable elementary submodels of $H(\theta)$ and $\cal{T}$ a set of transitive elementary submodels of $H(\theta)$.
\end{notation}

\begin{definition}\label{def:appropriate}
    We say that $\cal{S}$ and $\cal{T}$ are \emph{appropriate} for $H(\theta)$ if
    \begin{itemize}
        \item all elements of $\cal{S}\cup\cal{T}$ are members of $H(\theta)$, and
        \item if $W\in\cal{T}$, $M\in\cal{S}$, and $W\in M$, then $M\cap W$ is a member of $W$ and a member of $\cal{S}$.
    \end{itemize}
\end{definition}

We assume for the remainder of this section that $\cal{S}$ and $\cal{T}$ are appropriate for $H(\theta)$. We will review the definition of the two-type model sequence poset $\ps_{\cal{S},\cal{T},H(\theta)}$, which we abbreviate throughout this section by $\ps$. We first define which objects are conditions, and then after a few remarks, we define the ordering.

\begin{definition}\label{def:sequencePoset}
    Conditions in $\ps_{\cal{S},\cal{T},H(\theta)}$ are finite $\in$-increasing sequences of elements of $\cal{S}\cup\cal{T}$ which are closed under intersections.
\end{definition}

\begin{remark}\hfill
\begin{enumerate}
    \item To say that a sequence $\la M_k:k<m\ra$ is closed under intersections means that if $k,l<m$, then $M_k\cap M_l$ is also on the sequence.
    \item One of the key places where closure under intersections is used is the proof of properness for small nodes.
\end{enumerate}
\end{remark}

\begin{lemma}
    If $\la M_k:k<m\ra$ is a condition and $k<l<m$, then the von Neumann rank of $M_k$ is less than that of $M_l$.
\end{lemma}

Consequently, there is no loss of information in talking about the set $\lb M_k:k<m\rb$ whenever $\la M_k:k<m\ra$ is a condition in the model sequence poset. This helps to define the order:

\begin{definition}
    Conditions in $\ps$ are ordered by reverse inclusion: $\la N_l:l<n\ra\leq\la M_k:k<m\ra$ iff $\lb N_l:l<n\rb\supseteq\lb M_k:k<m\rb$.
\end{definition}

It will be helpful to have a slightly weaker notion than that of a condition, and the next item takes care of this:

\begin{definition}\label{def:precondition}
    A finite sequence $\la M_k:k<m\ra$ of elements of $\cal{S}\cup\cal{T}$ is a \emph{precondition} if it is $\in$-increasing.
\end{definition}

Thus conditions in $\ps$ are precisely  preconditions which are closed under intersections. However, it suffices to check a seemingly weaker condition than full closure under intersections to verify that a precondition is in fact a condition:

\begin{lemma}\label{lemma:weakclosure}
    Suppose that $s$ is a precondition and that $s$ satisfies the following weak closure property:
    \begin{enumerate}
        \item[$(*)$] if $W$ and $M$ are nodes in $s$ of transitive and small type respectively and if $W\in M$, then $M\cap W$ is in $s$. 
    \end{enumerate}
    Then $s$ is outright closed under intersections, and therefore $s$ is a condition.
\end{lemma}

We now state some results about amalgamating conditions in the model sequence poset.

\begin{definition}
    Let $s$ be a condition and $Q\in s$. The \emph{residue} of $s$ to $Q$, denoted $\operatorname{res}_Q(s)$, is the set $\lb M\in s:M\in Q\rb$.
\end{definition}

Note that if $s$ is a condition, so is $\operatorname{res}_Q(s)$ (Lemma 2.18 of \cite{Neemanttm}). It will be helpful to have the following additional terminology when discussing amalgamation arguments.

\begin{definition}
    Two conditions $s$ and $t$ are \emph{directly compatible} if the closure of $s\cup t$ under intersections is a condition.
\end{definition}

    Two compatible conditions needn't in general be directly compatible. For example, suppose that $W\in\cal{T}$ and $M\in\cal{S}$ where $W\in M$. Then $M\cap W$ is also a model in $\cal{S}$ (Definition \ref{def:appropriate}). Note that the conditions $\lb M\cap W\rb$ and $\lb M\rb$ are compatible since $\lb M\cap W,W,M\rb$ is a condition extending both. However, the closure of $\lb M\cap W,M\rb$ under intersections is just $\lb M\cap W,M\rb$ again, and this is not a condition since $M\cap W\notin M$ (this in turn because $M\cap W\in W$).

    In fact, any condition containing $M\cap W$ and $M$ must add a transitive node between these two models, as otherwise, we would have a finite $\in$-chain of \emph{countable} models from $M\cap W$ to $M$, which in turn would imply that $M\cap W\in M$.

The next lemma states the amalgamation result for the case when $Q$ is transitive:

\begin{lemma}
    Suppose that $s$ is a condition and that $Q\in s$ is transitive. Let $t\in Q$ be a condition extending $\operatorname{res}_Q(s)$. Then $s\cup t$ is a condition.
\end{lemma}

Note in this case that we don't need to close $s\cup t$ under intersections. However, the case when $Q$ is a small node is much more involved. It is here that the closure under intersections is crucial. The next item ``amalgamates" Lemma 2.21 and Corollary 2.31 of \cite{Neemanttm}.

\begin{lemma}
    Suppose that $s$ is a condition and $Q\in s$ a small node. Let $t\in Q$ be a condition which extends $\operatorname{res}_Q(s)$. Then $s$ and $t$ are directly compatible.

    Furthermore, if $r$ is the closure of $s\cup t$ under intersections, then
    \begin{itemize}
        \item $\operatorname{res}_Q(r)=t$, and
        \item the small nodes of $r$ outside of $Q$ are all of the form $N$ or $N\cap W$ where $N\in s$ is small and $W\in t$ is transitive.
    \end{itemize}
\end{lemma}

The next (rather comforting) lemma tells us that if a condition is inside a model, then we can add that model to the condition.

\begin{lemma}\label{lemma:tInM}
    Let $M\in\mathcal{S}\cup\mathcal{T}$, and let $t$ be a condition in $M$. Then there is a condition $r\leq t$ with $M\in r$. In fact, the closure of $t\cup\lb M\rb$ under intersections is such an $r$.
\end{lemma}

The following lemma (combining Claim 5.7 and Remark 5.8 of \cite{Neemanttm}) shows that under certain conditions, any transitive node can be added to any condition. We also include a helpful claim which extracts an additional feature of the proof. We will use these facts in later sections:

\begin{lemma}\label{lemma:addtransitivenodes}
    Suppose that $\theta$ is an inaccessible cardinal, and let $\la H(\theta);A_1,\dots,A_k\ra$ be an expansion of $H(\theta)$ by finitely many predicates. Suppose that $\cal{T}$ consists exactly of nodes $H(\al)$ so that $\al$ has uncountable cofinality (equivalently, $H(\al)$ is countably closed) and so that 
    $$
    \la H(\alpha);A_1\cap H(\alpha),\dots,A_k\cap H(\alpha)\ra\prec \la H(\theta);A_1,\dots,A_k\ra.
    $$
    Suppose further that every $M\in\cal{S}$ is elementary in $\la H(\theta);A_1,\dots,A_k\ra$. Then for any condition $s$ in $\ps_{\cal{S},\cal{T},H(\theta)}$ and any node $H(\al)\in\cal{T}$:
    \begin{enumerate}
        \item there is an $r\leq s$ with $H(\al)\in r$, and
        \item one can find such a condition $r$ where the new nodes in $r$ are either transitive or of the form $N\cap W$ where $N\in s$ and $W\in\cal{T}$.
    \end{enumerate}
\end{lemma}

The next lemma extracts an additional feature of the proof of Claim 5.7 and Remark 5.8 of \cite{Neemanttm}.

\begin{lemma}\label{lemma:addTopTransitive} (Under the same assumptions as Lemma \ref{lemma:addtransitivenodes}) Suppose that $s$ is a condition, $H(\al)\in \cal{T}$, and that $s\not\subseteq H(\alpha)$. Let $M$ be the least member of $s$ not in $H(\alpha)$, and suppose that $M\neq H(\alpha)$. Then
\begin{enumerate}
    \item $(M\cap\theta)\setminus\alpha\neq\emptyset$;
    \item if $\alpha\in M$, then there is an $r$ satisfying (1) and (2) of Lemma \ref{lemma:addtransitivenodes} so that the transitive nodes of $r$ are exactly $H(\alpha)$ and those of $s$;
    \item if $\alpha\notin M$, then setting $\alpha^*:= \min((M\cap\theta)\setminus\alpha)$, we have that $H(\alpha^*)$ is a member of $\cal{T}$.
\end{enumerate}
\end{lemma}

\section{Basic Features of the $\pfa_\Phi(X)$ Poset}\label{Section:Poset}

In this section, we define a ``$\pfa_\Phi(X)$-type" poset, where $X$ and $\Phi$ are parameters representing, respectively, a topological space whose size is below a given supercompact and a topological property of that space. In the case when $\Phi$ is either ``Lindel{\"o}f" or ``countably tight", we will show in later sections that forcing with the corresponding $\pfa_\Phi(X)$-type poset will give a model of $\pfa_\Phi(X)$. The purpose of this section is, after giving the definition, to highlight a number of general properties of this kind of poset. Most of our work in the current section involves translating Section 6 of \cite{Neemanttm}.

\begin{assumption}\label{GroundModelassumption}
    $V$ is a ground model of the $\gch$, $\theta$ is a supercompact cardinal in $V$, and $J:\theta\to H(\theta)$ is a Laver function for $\theta$. Finally, $X$ is a space of size $<\theta$ satisfying the topological property $\Phi$.
\end{assumption}

The forcing which we will use to create a model of $\mathsf{PFA}_\Phi(X)$ (see Definition \ref{def:axioms}) is a slight variation of the finite conditions poset from \cite{Neemanttm} which creates a model of $\mathsf{PFA}$.

We first define the permissible stages for the working part of a condition, and then we will define the sets $\cal{S}$ and $\cal{T}$ of nodes. Let $Z$ be the set of $\al<\theta$ so that $(H(\al),\in,J\upharpoonright\al,\lb X\rb)$ is elementary in $(H(\theta),\in,J,\lb X\rb)$. For each $\al\in Z$, define $\varphi(\al)$ to be the least cardinal so that $J(\al)\in H(\varphi(\al))$, and observe that $\varphi(\al)<\min(Z\bsl(\al+1))$.

Next, let $Z^*$ denote $Z\cap\operatorname{cof}(>\omega)$, and define $\cal{T}$ to be the set 
$$
\lb H(\al):\al\in Z^*\rb.
$$
Observe that if $\al\in Z$ has uncountable cofinality (i.e., if $\al\in Z^*$), then $H(\al)$ is countably closed. We define $\cal{S}$ to be the set of countable elementary submodels of $(H(\theta),\in,J,\lb X\rb)$. Note that $\cal{S}$ and $\cal{T}$ are appropriate for $H(\theta)$ (Definition \ref{def:appropriate}).

The next item contains the definition of the poset $\bb{P}_\Phi(X)$ which will give us a model of $\pfa_\Phi(X)$, at least for certain values of the parameters $X$ and $\Phi$. The definition is by recursion on $\al\in Z$, since we need to know $\bb{P}\cap H(\al)$ in order to evaluate certain names:

\begin{definition}\label{def:PFAiteration}
    Conditions in $\bb{P}_\Phi(X)$ (hereafter $\ps$) are pairs $\la s,p\ra$ satisfying the following:
    \begin{enumerate}
        \item $s$ is a condition in the model sequence poset $\ps_{\cal{S},\cal{T},H(\theta)}$ from Definition \ref{def:sequencePoset}.
        \item $p$ is a finite partial function whose domain is a (perhaps proper) subset of the set of $\al\in Z^*$ so that
        \begin{enumerate}
            \item $H(\al)\in s$, and
            \item $\ps\cap H(\al)$ forces that $J(\al)$ preserves that $X$ has property $\Phi$, and\footnote{We will later address minor technicalities which arise if $\ps\cap H(\al)$ does not force this condition (2b) or does not force the next condition (2c).}
            \item $\ps\cap H(\al)$ forces that $J(\al)$ is proper for all models of the form\footnote{Note that by $M[\dot{G}_\al]$ we mean the name for $M[G_\al]$, whenever $G_\al$ is a $V$-generic filter over $\ps\cap H(\al)$.} $M[\dot{G}_\al]$ where $M\in\cal{S}$ and $\al\in M$.
        \end{enumerate}
    
        \item For each $\al\in\dom(p)$, $p(\al)\in H(\varphi(\al))$.
        \item For each $\al\in\dom(p)$, 
        $$
        \Vdash_{\ps\cap H(\al)}p(\al)\in J(\al).
        $$
        \item For each $\al\in\dom(p)$ and for each small node $M\in s$ with $\al\in M$,
        $$
        \la s\cap H(\al),p\upharpoonright\al\ra\Vdash_{\ps\cap H(\al)} p(\al)\text{ is an }(M[\dot{G}_\al],J(\al))\text{-generic condition}.
        $$ 
    \end{enumerate}
    We order conditions as follows: $\la s^*,p^*\ra\leq\la s,p\ra$ iff 
    \begin{itemize}
        \item $s^*\leq s$ in the model sequence poset,
        \item $\dom(p)\seq\dom(p^*)$, and
        \item for all $\al\in\dom(p)$,
        $$
         \la s^*\cap H(\al),p^*\upharpoonright\al\ra\Vdash_{\ps\cap H(\al)} p^*(\al)\leq_{J(\al)}p(\al).
        $$
    \end{itemize}
\end{definition}

    The only difference of difference between our definition and that of \cite{Neemanttm} is that we require the posets used along the way to preserve the property $\Phi$ of $X$. Accordingly, certain proofs from \cite{Neemanttm} carry through without change to the case under consideration here. We summarize some of these in what follows, beginning with a few comments about condition (5).

\begin{remark}\label{remark:condition5initialsegment}
    Condition (5) of Definition \ref{def:PFAiteration} holds for $\al$ and $M$ iff it holds for $\al$ and $M\cap H(\ga)$ whenever $\ga\in Z$ is above $\al$ (Remark 6.4 of \cite{Neemanttm}).
\end{remark}
    
The following lemma (Claim 6.5 of \cite{Neemanttm}) gives a seemingly weaker version of (5) which is in fact equivalent and easier to check ``in practice."

\begin{lemma}\label{lemma:genericforInterval}
    Condition (5) in Definition \ref{def:PFAiteration} is equivalent to the same condition restricted to $M\in s$ which occur above $H(\al)$ and so that there are no transitive nodes of $s$ between $H(\al)$ and $M$.
\end{lemma}

Next, we introduce various restrictions of $\ps$ to $\be\in Z$ (not just $\beta\in Z^*$). The first of these will simply be $\ps\cap H(\be)$. The other of these restricts only the working part to be below $\be$, but places no constraints on the side conditions part (beyond those of Definition \ref{def:PFAiteration}). Having these notations will facilitate certain later inductive arguments. 

\begin{notation}\label{notation:restriction}
    Given $\be\in Z$, we let $\ps\upharpoonright\beta$ abbreviate $\ps\cap H(\beta)$, with the restriction of the ordering from $\ps$. Given $W\in\cal{T}$, say $W=H(\gamma)$, we use $\ps\upharpoonright W$ to abbreviate $\ps\upharpoonright\gamma$.

    We also let $\ps_\beta$ denote the poset with the same definition as $\ps$, but where the beginning of condition (2) is replaced by the statement ``...whose domain is a subset of the set of $\al\in Z^*\cap\beta$ so that..." This also inherits the ordering from $\ps$.
\end{notation}

The next item is Corollary 6.11 of \cite{Neemanttm}.

\begin{lemma}\label{lemma:addM}
    Let $M$ be a small node and $\la t,q\ra$ a condition in $\bb{P}\cap M$. Let $s$ be the closure of $t\cup\lb M\rb$ under intersections, which by Lemma \ref{lemma:tInM} is a condition in the model sequence poset. Then there is a $p$ so that $\la s,p\ra$ is a condition extending $\la t,q\ra$.
\end{lemma}


In the next few items, we review strong properness for transitive nodes and properness for small nodes. The following summarizes Claim 6.6 and Lemma 6.7 from \cite{Neemanttm}.

\begin{lemma}\label{lemma:transitivestronglygeneric}Let $\la s,p\ra$ be a condition in $\ps$ and $W=H(\al)$ a transitive node in $s$.
    \begin{enumerate}
        \item If $\la t,q\ra$ extends $\la s,p\ra\upharpoonright W$ in $\ps\upharpoonright W$, then $\la s,p\ra$ and $\la t,q\ra$ are compatible.
        \item  $\la s,p\ra$ is a strongly $(W,\ps)$-generic condition.
        \item $\ps$ is strongly proper for $\cal{T}$.
    \end{enumerate}
\end{lemma}

Before addressing properness for small nodes, we will prove that we can always add any transitive node to any condition. Along the way, we need to establish a bit of notation for dealing with the cases when $\ps\cap H(\al)$ does not force that the conditions in (2b) and (2c) of Definition \ref{def:PFAiteration} obtain. Towards this end, we introduce the following notation:

\begin{notation}\label{notation:weirdJ}
    Define $J':\theta\to H(\theta)$ to be the function where for $\al\in Z^*$, $J'(\al)=J(\al)$ if $\ps\cap H(\al)$ forces (2b) and (2c) of Definition \ref{def:PFAiteration} to hold and where $J'(\al)$ is the trivial poset otherwise.
\end{notation}

Having this notation will streamline some later discussions. Now for the lemma on adding transitive nodes:

\begin{lemma}\label{lemma:addtransitivenodesinPoset}
Let $\la s,p\ra$ be a condition and $\al\in Z^*$. Then there is an extension $s^*$ of $s$ with $H(\al)\in s^*$ so that $\la s^*,p\ra$ is a condition.

Furthermore, if $J'(\al)=J(\al)$ (i.e., if $J(\al)$ names an appropriate poset), then there is a finite partial function $p^*$ with $\al\in\dom(p^*)$ so that $\la s^*,p^*\ra$ is a condition extending $\la s,p\ra$.
\end{lemma}
\begin{proof}
    Fix $\la s,p\ra$, and suppose that $H(\al)\notin s$. Note that the assumptions of Lemma \ref{lemma:addtransitivenodes} are satisfied in the present context. Let $s^*\leq s$ be an extension in the model sequence poset with $H(\al)\in s^*$ so that (2) of Lemma \ref{lemma:addtransitivenodes} is satisfied with respect to $s^*$, i.e., so that the new nodes in $s^*$ are either transitive or of the form $N\cap W$ with $N\in s$ and $W\in\cal{T}$. By Remark \ref{remark:condition5initialsegment}, $\la s^*,p\ra$ is a condition.

    Now suppose that, in addition, $J'(\al)=J(\al)$, but that $\al\notin\dom(p)$. Let $W$ be the least transitive node of $s^*$ above $H(\al)$ is there is one, and let $W$ equal $H(\theta)$ otherwise. Let $M_0\in M_1\in \dots\in M_{m-1}$ enumerate in $\in$-increasing order the interval of small nodes of $s^*$ between $H(\al)$ and $W$. By Lemma \ref{lemma:genericforincreasingsequence}, we can find a $(\ps\upharpoonright\al)$-name $\dot{u}_\al$ for a condition in $J(\al)$ which is forced to be $(M_i,J(\al))$-generic for each $i<m$.

    Define $p^*:=p\cup\lb \al\mapsto\dot{u}_\al\rb$. We claim that $\la s^*,p^*\ra$ is a condition. The only item that remains to be checked is condition (5) of Definition \ref{def:PFAiteration} at the point $\al$. However, by construction of the name $\dot{u}_\al$, we know that the assumptions of Lemma \ref{lemma:genericforInterval} are satisfied. Hence condition (5) of Definition \ref{def:PFAiteration} is outright satisfied. This completes the proof that $\la s^*,p^*\ra$ is a condition.
\end{proof}

In the following item, we state Lemma 6.12 from \cite{Neemanttm}, the key lemma on properness for small nodes. Recall the notation $\ps_\beta$ from Notation \ref{notation:restriction}.

\begin{lemma} Let $\be\in Z\cup\lb\theta\rb$. Fix a regular cardinal $\theta^*>\theta$ and a countable $M^*\prec H(\theta^*)$ containing $\theta$, $J$, $X$, and $\beta$ as elements. Let $M:=M^*\cap H(\theta)$, noting that $M\in\cal{S}$. Finally let $\la s,p\ra$ be a condition in $\ps_\be$ with $M\in s$. Then
\begin{enumerate}
    \item for each $D\in M^*$ which is dense in $\ps_\be$, there exists $\la t,q\ra\in D\cap M^*$ which is compatible with $\la s,p\ra$. Moreover, we can find $\la s^*,p^*\ra\in\ps_\be$ extending both $\la s,p\ra$ and $\la t,q\ra$ so that $\operatorname{res}_M(s^*)\bsl H(\be)\seq t$ and so that every small node of $s^*$ above $\be$ but outside of $M$ is either a node of $s$ or of the form $N'\cap W$ where $N'$ is a small node in $s$ and $W\in\cal{T}$.
    \item $\la s,p\ra$ is an $(M^*,\ps_\be)$-generic condition.
\end{enumerate}
\end{lemma}

In the above lemma, (2) follows from applying (1) to all extensions of $\la s,p\ra$. The proof of (1) is by induction on $\be$. The ``moreover" part of (1) gives additional information on a condition witnessing compatibility, namely, that such a condition can be found which adds new nodes in a very controlled way, at least above $\beta$. Naturally enough, this is used in the inductive argument.

Later in the present work, we will need a more local version of this lemma. The proof is almost word-for-word identical to the proof of the above in Neeman's paper, and hence we will not reproduce the proof here.


\begin{lemma}\label{lemma:localproper}
    Let $\al\leq\be<\ga$ be ordinals in $Z\cup\lb\theta\rb$, where $\be$ has uncountable cofinality.  Fix a countable $M^*\prec H(\ga)$ so that $M^*$ contains $\al$, $\be$, $X$, and $J\upharpoonright H(\be)$ as elements. Let $M:=M^*\cap H(\be)$, noting that $M$ is in $\cal{S}$. Finally, let $\la s,p\ra$ be a condition in $\ps_\al\cap H(\be)$ with $M\in s$. Then
    \begin{enumerate}
        \item for each $D\in M^*$ which is dense in $\ps_\al\cap H(\beta)$, there exists $\la t,q\ra\in D\cap M^*$ which is compatible with $\la s,p\ra$. Moreover, we can find $\la s^*,p^*\ra\in\ps_\al\cap H(\beta)$ extending $\la s,p\ra$ and $\la t,q\ra$ so that $\operatorname{res}_M(s^*)\bsl H(\alpha)\seq t$ and so that every small node of $s^*$ above $\alpha$ but outside of $M$ is either a node of $s$ or of the form $N'\cap W$ where $N'$ is a small node in $s$ and $W\in\cal{T}$.
        \item $\la s,p\ra$ is an $(M^*,\ps_\al\cap H(\beta))$-generic condition.
    \end{enumerate}

\end{lemma}

The following lemma summarizes Corollary 6.13 of \cite{Neemanttm}:

\begin{lemma}\label{lemma:cardinalstructure}
    The poset $\bb{P}=\bb{P}_\Phi(X)$ is proper, and hence preserves $\om_1$. $\ps$ also preserves all cardinals $\lambda\geq\theta$ and collapses all cardinals between $\om_1$ and $\theta$ to have size $\om_1$. Hence $\ps$ forces that $\theta$ becomes $\aleph_2$.
\end{lemma}

\section{Embedding Lemmas for the $\pfa_\Phi(X)$ Poset}\label{Section:EmbeddingLemmas}

In this section, we will define the side conditions augmentation of a given poset. We will use these ideas in our later inductive proof that $\ps_\Phi(X)$ preserves the given topological property $\Phi$ of $X$ (for $\Phi=$ ``countably tight" or ``Lindel{\"o}f"). Indeed, one of the subcases within the successor case of our later inductive argument will involve $\al\in Z^*$ and $\be$ the next element of $Z^*$ above $\al$; we will argue that $\ps\cap H(\be)$ is equivalent to $\ps\cap H(\al)$ followed by not just $J'(\al)$ itself, but by $J'(\al)$ augmented with certain countable models (recall $J'(\al)$ from Notation \ref{notation:weirdJ}). That is the main goal of this section.


\begin{definition}\label{def:sideaugment}
Let $\bb{A}$ be a poset, $\mu$ a large enough regular cardinal, and $\cal{M}$ a stationary set of countable elementary submodels of $H(\mu)$, so that for each $M\in\cal{M}$, $\bb{A}\in M$ and $\bb{A}$ is proper for $M$. The \emph{$\cal{M}$-augmentation of $\bb{A}$}, denoted $\cal{M}(\bb{A})$, is the poset consisting of pairs $\left<a,p \right>$ where $a\in{\bb A}$, $p$ is a finite $\in$-chain of elements of $\cal{M}$, and for each model $M\in p$, the condition $a$ is an $(M,{\bb A})$-generic condition.

We order conditions as follows: $\la a_1,p_1\ra$ extends $\la a_0,p_0\ra$, written $\la a_1,p_1\ra\leq\la a_0,p_0\ra$, if $a_1\leq_{\bb{A}}a_0$ and $p_0\seq p_1$.
\end{definition}

We will show in the next section that if $\bb{A}$ preserves that $X$ is Lindel{\"o}f (resp. countably tight) then $\cal{M}(\bb{A})$ also preserves that $X$ is Lindel{\"o}f (resp. countably tight). We will leverage this in the successor case of our later inductive argument that $\ps_\Phi(X)$ preserves $\Phi$ of $X$ for certain values of these parameters. In this section, we are concerned to show how the side conditions augmentation arises naturally in the study of $\ps_\Phi(X)$.

Recalling  the notation $J'$ from Notation \ref{notation:weirdJ},  our goal is to show that if $\al\in Z^*$ and $\be$ is the least element of $Z^*$ above $\al$, then there is a dense embedding from (a dense subset of) $\ps\upharpoonright\be$ to the two-step poset of $\ps\upharpoonright\al$ followed by a certain side conditions augmentation of $J'(\al)$. Note that in the case that $J'(\al)$ names the trivial poset, this side conditions augmentation will be (isomorphic to) forcing with finite $\in$-chains of countable models. We begin by defining the dense set which will be the domain of this embedding.

\begin{definition}
    Let $\al\in Z^*$, and let $\be$ be the least element of $Z^*$ above $\al$. Let $D_{\al,\be}$ consist of all conditions $\la s,p\ra\in \ps\upharpoonright\be$ so that $H(\al)\in s$ and so that if $J'(\al)=J(\al)$, then $\al\in \dom(p)$.
\end{definition}

\begin{lemma}
    $D_{\al,\be}$ is dense in $\ps\upharpoonright\be$.
\end{lemma}
\begin{proof}
    By Lemma \ref{lemma:addtransitivenodesinPoset}.
\end{proof}

Now we define the models that are appropriate for the relevant side conditions augmentation.

\begin{definition}\label{def:modelsforSCAug}
    Let $\al$ and $\be$ be as before. Let $\dot{\cal{M}}_\al^0$ be the $(\ps\upharpoonright\al)$-name defined as follows: given an arbitrary $V$-generic filter $G_\al$ over $\ps\upharpoonright\al$, let
    $$
    \dot{\cal{M}}_\al^0[G_\al]:=\lb M\in H(\be)\cap\cal{S}:(\exists\la s,p\ra\in D_{\al,\be})\;\left[\la s,p\ra\upharpoonright\al\in G_\al\we M\in s\we\al\in M\right]\rb.
    $$
    Also let $\dot{\cal{M}}_\al$ be the name defined by
    $$
    \dot{\cal{M}}_\al[G_\al]:=\lb M[G_\al]:M\in\dot{\cal{M}}^0_\al[G_\al]\rb.
    $$
\end{definition}

In the previous definition, we are making the distinction between $\cal{M}_\al^0$ and $\cal{M}_\al$ in order to facilitate a later proof that the map taking $M\in\cal{M}^0_\al$ to $M[G_\al]\in\cal{M}_\al$ is an $\in$-isomorphism.

We'll now work to establish that $\ps\upharpoonright\be$ is forcing equivalent to $\ps\upharpoonright\al$ followed by the side conditions augmentation of $J'(\al)$ by $\dot{\cal{M}}_\al$. Recalling Definition \ref{def:sideaugment}, we first establish the stationarity of $\dot{\cal{M}}_\alpha$:

\begin{lemma}
    $\ps\upharpoonright\al$ forces that $H^{V[\dot{G}_\al]}(\be)=H(\be)[\dot{G}_\al]$ and that $\dot{\cal{M}}_\al$ is stationary in $H^{V[\dot{G}_\al]}(\be)$.
\end{lemma}
\begin{proof}
    The equality claimed in the lemma follows from the fact that $\ps\upharpoonright\al$ has the $\al^+$ chain condition; indeed it has size $\al$ on account of assuming the $\gch$ (see Assumption \ref{GroundModelassumption}).

    For the second item in the statement of the lemma, fix a name $\dot{F}$ for a function from finite subsets of $H^{V[\dot{G}_\al]}(\be)$ to $H^{V[\dot{G}_\al]}(\be)$, and fix a condition $\la \bar{s},\bar{p}\ra$ in $\ps\upharpoonright\al$. We find a model $M$ and an extension of $\la\bar{s},\bar{p}\ra$ which forces that $M[\dot{G}_\al]$ is in $\dot{\cal{M}}_\al$ and is closed under $\dot{F}$.

    Indeed, let $M^*\prec H(\theta)$ be countable with all of the parameters of interest, including $\dot{F}$, $\al$, and $\la\bar{s},\bar{p}\ra$. Let $M:=M^*\cap H(\be)$, so that $M\in\cal{S}$. Apply Lemma \ref{lemma:addM} to fix a condition $\la s',p'\ra$ in $\ps\upharpoonright\be$ extending $\la \bar{s},\bar{p}\ra$ so that $M\in s'$. Then $\la s',p'\ra\upharpoonright\al$ forces that $M[\dot{G}_\al]$ is in $\dot{\cal{M}}_\al$ (with $\la s',p'\ra$ providing the witness). Finally, if $G_\al$ is a $V$-generic filter over $\ps\upharpoonright\al$, then $M[G_\al]$ is elementary in $H(\be)[G_\al]=H^{V[G_\al]}(\be)$, and since $F:=\dot{F}[G_\al]$ is in $M[G_\al]$, $M[G_\al]$ is closed under $F$.
\end{proof}

The following lemma uses standard facts about properness:

\begin{lemma}\label{lemma:membershipbijection}
    $\ps\upharpoonright\al$ forces that the map taking $M\in\dot{\cal{M}}_\al^0$ to $M[\dot{G}_\al]\in\dot{\cal{M}}_\al$ is an $\in$-isomorphism.
\end{lemma}
\begin{proof}
    Fix a $V$-generic filter $G_\al$ over $\ps\upharpoonright\alpha$. The map taking $M$ to $M[G_\al]$ is a surjection by definition of $\cal{M}_\al$.

    Next, fix $M_0\in M_1$, both in $\cal{M}^0_\al$, and note that $\bb{P}\upharpoonright\al$ is a member of both models. Thus $M_1[G_\al]$, as it contains $G_\al$ and $M_0$, contains $M_0[G_\al]$ as an element. Therefore the map preserves membership from the left to the right.

    We now show that this map is an injection, a part of the argument that uses properness. Fix $M$ and $N$ in $\cal{M}^0_\al$, and suppose that $M\neq N$. Apply the definition of $\cal{M}^0_\al$ and Lemma \ref{lemma:localproper} (with the present $\al$ playing the twin roles of the ``$\al$" and ``$\be$" in the statement of that lemma, and with the present $\be$ playing the role of the ``$\ga$" in the statement of that lemma) to conclude that $G_\al$ has generic conditions for $M$ and for $N$ in the poset $\ps\upharpoonright\alpha$. Consequently, $M[G_\al]\cap V=M$ and $N[G_\al]\cap V=N$. Since $M[G_\al]$ and $N[G_\al]$ have distinct traces on $V$, they are distinct models. This completes the proof that the stated map is an injection.
    

    Now suppose that we have two models $M_0$ and $M_1$ in $\cal{M}^0_\al$ and that $M_0[G_\al]\in M_1[G_\al]$; we need to show that $M_0\in M_1$. Because $M_1[G_\al]\cap V=M_1$ (applying properness as in the previous paragraph), it suffices to argue that $M_0\in M_1[G_\al]$. Towards this end, let $\ga:=\sup(M_0[G_\al]\cap\text{On})$, noting that $\ga\in M_1[G_\al]$. Since $\ga\in M_1[G_\al]\cap\text{On}$, $\ga\in M_1$. Hence $H^V(\ga)\in M_1$.
    
    Next, note that $M_0[G_\al]\cap\text{On}=M_0\cap\text{On}$. Thus $M_0\seq H^V(\ga)$, since $M_0$ is closed under the map taking $x$ to the cardinality of the transitive closure of $x$. Since $H^V(\ga)$ and $M_0[G_\al]$ are both in $M_1[G_\al]$, so is
    $$
    M_0[G_\al]\cap H^V(\ga)=M_0\cap H^V(\ga)=M_0,
    $$
    completing the proof.
\end{proof}

The following notation will be helpful for the next proposition.

\begin{notation}
    Given $\la s,p\ra\in D_{\al,\be}$ we let $s\bsl\al$ denote $\lb M\in s:\al\in M\rb$. Note that $H(\al)\in s$ by definition of $D_{\al,\be}$ and that each node in $s$ above $H(\al)$ is a small node, since $\be$ is the next element of $Z^*$ above $\al$. 
    
    We next let $(s\bsl\al)[\dot{G}_\al]$ be the name for $\lb M[\dot{G}_\al]:M\in s\bsl\al\rb$.
\end{notation}

The next proposition is the main result of this section, specifying the desired dense embedding; on account of $J'(\al)$ being a nuisance, we will need to introduce yet more notation in the statement of the proposition.

\begin{proposition}\label{prop:denseembedding}
    Let $\la s,p\ra\in D_{\al,\be}$ be a condition. Define $p'(\al)$ to be $p(\al)$ if $J'(\al)=J(\al)$, and otherwise set $p'(\al)$ to be the trivial condition in the trivial poset.
    
    Let $\Gamma$ be the map taking $\la s,p\ra$ in $D_{\al,\be}$ to 
    $$
    \left(\la s,p\ra\upharpoonright\al,\la p'(\al),(s\bsl\al)[\dot{G}_\al]\ra\right).
    $$ 
    Then $\Gamma$ is a dense embedding from $D_{\al,\be}$ to $(\ps\upharpoonright\al)\ast(\dot{\cal{M}}_\al(J'(\al)))$.
\end{proposition}
\begin{proof}
    It is clear that $\Gamma$ preserves the order from left to right. 

    Next we need to show that if the $\Gamma$-images of two conditions $\la s_1,p_1\ra$ and $\la s_2,p_2\ra$ are compatible on the right hand side, then the $\la s_i,p_i\ra$ are compatible in $D_{\al,\be}$.

    Thus fix a condition $\left(\la\bar{s},\bar{p}\ra,\la \dot{b},\dot{m}\ra\right)$ in the right hand side poset which extends the $\Gamma$-images of each $\la s_i,p_i\ra$. We will extend $\la\bar{s},\bar{p}\ra$ in a few stages (but shamelessly keep writing $\la\bar{s},\bar{p}\ra$). By Lemma \ref{lemma:membershipbijection}, the map taking $M\in\dot{\cal{M}}^0_\al$ to $M[\dot{G}_\al]\in\dot{\cal{M}}_\al$ is forced to be an $\in$-isomorphism. Accordingly, we may extend $\la\bar{s},\bar{p}\ra$ and relabel to assume that for some $\in$-increasing sequence $M_0\in\dots\in M_l$ in $V$, $\la\bar{s},\bar{p}\ra$ forces that $\dot{m}$ equals $\lb M_k[\dot{G}_\al]:k\leq l\rb$. Note, then, that each $s_i\bsl\al$ is a subset of $\lb M_k:k\leq l\rb$, again using the $\in$-isomorphism.
    
    Next, by definition of $\dot{\cal{M}}_\al$ (see Definition \ref{def:modelsforSCAug}), for each $M_k$, we may assume that $\la\bar{s},\bar{p}\ra$ extends some condition $\la t_k,q_k\ra\upharpoonright\al$, where $\la t_k,q_k\ra\in D_{\al,\be}$, and where $M_k\in t_k$. Note that since $t_k$ is a condition in the model sequence poset and $H(\al)\in t_k$ (because $\la t_k,q_k\ra\in D_{\al,\be}$), we have $M_k\cap H(\al)\in t_k\cap H(\al)$, and hence it is a member of $\bar{s}$.

    Now define $s$ to be $\bar{s}\cup\lb H(\al)\rb\cup\lb M_k:k\leq l\rb$ and $p$ to be $\bar{p}\cup\lb\al\mapsto \dot{b}\rb$. We claim that $\la s,p\ra$ is a condition in $\ps\upharpoonright\be$ which extends each $\la s_i,p_i\ra$.

    First note that $s$ is an $\in$-chain. This is because all of the nodes in $\bar{s}$ are members of $H(\al)$,  $H(\al)$ is a member of $M_0$, and the $M_k$'s form an $\in$-chain. Next, we check that $s$ is closed under intersections, by applying Lemma \ref{lemma:weakclosure}. Therefore fix $W\in s$ transitive and $M\in s$ small, with $W\in M$, and we show that $W\cap M$ is in $s$. This holds when both $W$ and $M$ are below $H(\al)$ since $\bar{s}$ is a condition. This also holds when $W=H(\al)$, because then $M$ must equal one of the $M_k$, and $H(\al)\cap M_k\in\bar{s}$, as we noted earlier. Finally, if $W$ is below $H(\al)$ but $M$ is one of the $M_k$, then apply the fact that $M\cap H(\al)\in \bar{s}$ and the closure of $\bar{s}$ under intersections.

    What remains to be checked to see that $\la s,p\ra$ is a condition is that $\la s,p\ra\upharpoonright\alpha=\la\bar{s},\bar{p}\ra$ forces that $p(\al)$ is a generic condition in $J'(\al)$ for each of the $M_k$. But this follows since $p(\al)$ is the name $\dot{b}$ and by definition of the side conditions augmentation poset.

    We now finish the proof that the $\la s_i,p_i\ra$ are compatible in $D_{\al,\be}$. We know that $\la\bar{s},\bar{p}\ra$ extends each $\la s_i,p_i\ra\upharpoonright\alpha$. Next, $s$ contains all of the models in each $s_i$ at and above $H(\alpha)$. It remains to show that if $J'(\al)=J(\al)$, then  $\la s,p\ra\upharpoonright\alpha=\la\bar{s},\bar{p}\ra$ forces that $p(\al)\leq_{J(\al)} p_i(\alpha)$ for each $i$. But this follows from the fact that $\left(\la\bar{s},\bar{p}\ra,\la \dot{b},\dot{m}\ra\right)$ extends the $\Gamma$-images of the $\la s_i,p_i\ra$ and from the fact that $p(\al)$ is the name $\dot{b}$.

    To complete the proof that the map $\Gamma$ is a dense embedding, we need to show that the image is a dense subset of the poset on the right. Thus fix a condition $\left(\la\bar{s},\bar{p}\ra,\la \dot{b},\dot{m}\ra\right)$ in the poset on the right, and we find a condition in $\ps\upharpoonright\beta$ which maps below it. However, this follows by the arguments in the previous paragraphs, where we first extended $\la\bar{s},\bar{p}\ra$ and then constructed $\la s,p\ra$.
\end{proof}

\section{Preservation Properties of the Side Conditions Augmentation}\label{Section:SCAPreservation}

In this section, we will show that if $\bb{A}$ is a poset preserving that $X$ is Lindel{\"o}f (resp. countably tight) then the side conditions augmentation of $\bb{A}$ by $\cal{M}$ also preserves that $X$ is Lindel{\"o}f (resp. countably tight).

\begin{lemma}\label{lemma:SCAlindelof}
    Let $\bb{A}$, $\mu$, and $\cal{M}$ be as in Definition \ref{def:sideaugment}. Suppose that $(X,\tau)\in H(\mu)$ is a Lindel{\"o}f space and that $\bb{A}$ preserves that $X$ is Lindel{\"o}f. Then $\cal{M}(\bb{A})$ preserves that $X$ is Lindel{\"o}f.
\end{lemma}
\begin{proof}
    Recall first that it suffices to verify the Lindel{\"o}f property for open covers consisting of basic open sets and that $\tau$ is a basis for the topology on $X$ in the extension by $\cal{M}(\bb{A})$.
    
    Suppose that $\dot{\mathcal{U}}$ is an $\cal{M}(\bb{A})$-name for an open cover of $X$ where
    $$
    \Vdash_{\cal{M}(\bb{A})}\dot{\mathcal{U}}\subset\check{\tau}.
    $$
    Let $\la a_0,p_0\ra$ be a given condition. Fix a regular $\ka>\mu$. By applying the stationarity of $\cal{M}$ (see Definition \ref{def:sideaugment}), find a countable $M\prec H(\kappa)$ so that
    \begin{itemize}
        \item $M$ contains the parameters $\mu$, ${\bb A}$, $\cal{M}$, $\dot{\mathcal{U}}$, $\la a_0,p_0\ra$, and $(X,\tau)$, and
        \item $\bar{M}:=M\cap H(\mu)$ is in $\cal{M}$.
    \end{itemize}
    Set $p_1:=p_0\cup\lb \bar{M}\rb$, noting that this is an $\in$-chain since $p_0\subseteq \bar{M}$. Extend $a_0$ to an $(\bar{M},{\bb A})$-generic condition $a_1$, using the fact that $\bb{A}$ is proper for $\bar{M}$. Note that because $\bar{M}\cap\bb{A}=M\cap\bb{A}$ (on account of the fact that $\bb{A}\subseteq H(\mu)$), we have that $a_1$ is an $(M,\bb{A})$-generic condition. Note also that $\la a_1,p_1\ra$ extends $\la a_0,p_0\ra$ in $\cal{M}(\bb{A})$.

    Let $M(\dot{\mathcal{U}})$ be the following $\cal{M}(\bb{A})$-name:
    $$
    M(\dot{\mathcal{U}}):=\lb (\check{B},\la\bar{a},\bar{p}\ra)\in M:\la\bar{a},\bar{p}\ra\in\cal{M}(\bb{A})\we B\in M\cap\tau\wedge\la\bar{a},\bar{p}\ra\Vdash\check{B}\in\dot{\mathcal{U}}\rb.
    $$
    We claim that
    $$
    \la a_1,p_1\ra\Vdash M(\dot{\mathcal{U}})\text{ is a countable subcover of }\dot{\mathcal{U}}.
    $$
    Note that it is immediate that $\la a_1,p_1\ra$ forces that $M(\dot{\mathcal{U}})$ is countable, since $M$ is countable and $M(\dot{\mathcal{U}})$ is a subset of $M$.

    We first prove a claim that gives some necessary flexibility in dealing with extensions in the $\in$-chains. After this, we argue that no extension of $\la a_1,p_1\ra$ can force a specific counterexample.

    For each $\in$-chain $s$ in $M$ of models in $\cal{M}$ with $s\supseteq p_0$, let $\dot{\mathcal{U}}_s$ be the following ${\bb A}$-name:
    $$
    \dot{\mathcal{U}}_s:=\lb\la\check{B},\bar{a}\ra:B\in\tau \we \bar{a}\in{\bb A}\we\left(\exists\bar{p}\supseteq s\right)\;\left[\la\bar{a},\bar{p}\ra\Vdash_{\cal{M}(\bb{A})}\check{B}\in\dot{\mathcal{U}}\right]\rb.
    $$
    Observe that $\dot{\mathcal{U}}_s$ is a member of $M$, being definable in $H(\kappa)$ from parameters in $M$.\\
    
    \textbf{Claim 1:} Let $s\in M$ with $s\supseteq p_0$, and let $a\leq a_1$ be a condition in $\bb{A}$ which is $(N,\bb{A})$-generic for all $N$ in $s$. Then $a$ forces that $\dot{\mathcal{U}}_s$ is an open cover of $X$.\\

    \emph{Proof of Claim 1:}  To see this, let $G$ be an arbitrary $V$-generic filter over ${\bb A}$ containing $a$. Then $\mathcal{U}_s:=\dot{\mathcal{U}}_s[G]$ is a member of $M[G]$. Because $M[G]\prec H(\kappa)[G]$, it suffices to argue that $M[G]$ satisfies that $\mathcal{U}_s$ is an open cover of $X$. Towards this end, fix $x\in M[G]\cap X$. Since $a_1$ is an $(M,{\bb A})$-generic condition (as we remarked earlier) and $a_1\in G$, we know that $M[G]\cap V=M$, and in particular, $M[G]\cap X=M\cap X$. Thus $x\in M$.

    We now apply a density argument: consider the set $D_{x,s}$ of conditions $\bar{a}\in{\bb A}$ so that either
    \begin{itemize}
        \item no extension of $\bar{a}$ is generic for all models in $s$ or
        \item there exist $\bar{p}\supseteq s$ and  $B\in\tau$ with $x\in B$ so that $\la\bar{a},\bar{p}\ra\in \cal{M}(\bb{A})$ and
        $$
        \la\bar{a},\bar{p}\ra\Vdash_{\cal{M}(\bb{A})}\check{B}\in\dot{\mathcal{U}}.
        $$ 
    \end{itemize}
    $D_{x,s}$ is a member of $M$ by elementarity, and furthermore, $D_{x,s}$ is dense in ${\bb A}$. Again applying the fact that $a_1\in G$ is an $(M,{\bb A})$-generic condition, we can find some $\bar{a}\in D_{x,s}\cap G\cap M$. Since $\bar{a}$ and $a$ are both in $G$, they are compatible. Hence, the membership of $\bar{a}$ in $D_{x,s}$ must hold on account of the second option above. In light of this, and by applying the elementarity of $M$, we may find a $B\in M$ and a $\bar{p}\in M$ with $\bar{p}\supseteq s$ so that $\la\bar{a},\bar{p}\ra$ forces $\check{B}\in\dot{\mathcal{U}}$. Thus $\la\check{B},\bar{a}\ra$ is in $\dot{\cal{U}}_s$. Finally, because $\bar{a}\in G$, we see that $B\in\mathcal{U}_s$. As $x\in M[G]\cap X$ was arbitrary, this completes the proof of the claim. $\hfill\blacksquare\text{(Claim 1)}$\\

    The next claim uses that ${\bb A}$ preserves that $X$ is Lindel{\"o}f:\\
    
    \textbf{Claim 2:} Let $s\in M$ with $s\supseteq p_0$, and let $a\leq a_1$ be a condition in $\bb{A}$ which is $(N,\bb{A})$-generic for all $N$ in $s$. Then $a$ forces that\footnote{Note here that we do literally mean $M\cap\dot{\cal{U}}_s$, namely all pairs $\la\check{B},\bar{a}\ra\in\dot{\cal{U}}_s$ which are also members of $M$.} $M\cap\dot{\cal{U}}_s$ is an open cover of $X$.\\

    \emph{Proof of Claim 2:} Fix a $V$-generic filter $G$ over $\bb{A}$ containing $a$. By the previous claim, $\mathcal{U}_s$ is an open cover of $X$. Since $X$ is Lindel{\"o}f in $V[G]$, $\mathcal{U}_s$ has a countable subcover. However, $\mathcal{U}_s\in M[G]$, and so by the elementarity of $M[G]$, we can find a countable subcover $\mathcal{B}_s\in M[G]$. Note that $\mathcal{B}_s\seq M[G]$.

    Fix $x\in X$, and using that $\cal{B}_s$ is a cover, find $B\in\cal{B}_s$ with $x\in B$. Then $B\in M[G]\cap\mathcal{U}_s$, and so we can apply the elementarity of $M[G]$ to find $\bar{a}\in M[G]\cap G$ and $\bar{p}\supseteq s$ with $\bar{p}\in M[G]$ so that $\la\bar{a},\bar{p}\ra\Vdash\check{B}\in\dot{\mathcal{U}}$. Since $a_1$ is an $(M,{\bb A})$-generic condition, the objects $B$, $\bar{a}$, and $\bar{p}$ are all in $M[G]\cap V=M$. Thus the pair $\la\check{B},\bar{a}\ra$ is a member of the name $M\cap\dot{\cal{U}}_s$ 
    as witnessed by $\bar{p}$. Finally, $\bar{a}\in G$, and therefore $B\in (M\cap\dot{\cal{U}}_s)[G]$; since $x\in B$, this completes the proof. $\hfill\blacksquare\text{(Claim 2)}$\\

    We now put these pieces together, recalling that our goal is to show that $\la a_1,p_1\ra$ forces that $M(\dot{\mathcal{U}})$ is a countable subcover of $\dot{\mathcal{U}}$. Towards this end, fix an extension $\la a_2,p_2\ra\leq\la a_1,p_1\ra$ and an arbitrary $x\in X$. Let $s:=p_2\cap M$, noting that $s$ is an $\in$-chain in $M$ and that $s$ extends $p_0$. Since $a_2\leq a_1$ is generic for all of the models on $s$, $a_2$ forces that $M\cap\dot{\cal{U}}_s$ is an open cover of $X$. Accordingly, we can extend $a_2$ to a condition $a_3$ and select specific $B\in M\cap\tau$ with $x\in B$, $\bar{a}\in{\bb A}\cap M$ with $a_3\leq\bar{a}$, and $\bar{p}\supseteq s$ with $\bar{p}\in M$ so that $\la\bar{a},\bar{p}\ra\Vdash\check{B}\in\dot{\mathcal{U}}$. We also have that $p_3:=\bar{p}\cup p_2$ is an $\in$-chain: because $M\cap H(\mu)=\bar{M}$ and $\cal{M}\seq H(\mu)$, we have that $\bar{p}$ extends $s=p_2\cap M=p_2\cap\bar{M}$ and that $\bar{p}$ is in fact a member of $\bar{M}$. Additionally, as $\bar{a}$ is generic for the models in $\bar{p}$ and as $a_2$ is generic for the models in $p_2$, $a_3$ is generic for the models in $p_3=\bar{p}\cup p_2$. Thus $\la a_3,p_3\ra$ is a condition which extends both $\la a_2,p_2\ra$ and $\la\bar{a},\bar{p}\ra$. Since $\la\bar{a},\bar{p}\ra$ and $B$ are members of $M$ and since $\la\bar{a},\bar{p}\ra$ forces that $\check{B}\in\dot{\mathcal{U}}$, this shows that $\la a_3,p_3\ra$ forces that $\check{B}\in M(\dot{\mathcal{U}})$, which finishes the proof.
\end{proof}

\begin{lemma}\label{lemma:SCAtight}
    Let $\bb{A}$, $\mu$, and $\cal{M}$ be as in Definition \ref{def:sideaugment}. Suppose that $(X,\tau)\in H(\mu)$ is a countably tight space and that $\bb{A}$ preserves that $X$ is countably tight. Then $\cal{M}(\bb{A})$ preserves that $X$ is countably tight.
\end{lemma}
\begin{proof}
    Fix $\la a_0,p_0\ra$ in $\cal{M}(\bb{A})$, $x\in X$, and an $\cal{M}(\bb{A})$-name $\dot{E}$ so that
    $$
    \la a_0,p_0\ra\Vdash \check{x}\in\operatorname{cl}(\dot{E}).
    $$
    We will find an $\cal{M}(\bb{A})$-name $\dot{E}_0$ and a condition $\la a^*,p^*\ra$ extending $\la a_0,p_0\ra$ so that
    $$
    \la a^*,p^*\ra\Vdash\dot{E}_0\in\left[\dot{E}\right]^{\aleph_0}\we \check{x}\in\operatorname{cl}(\dot{E}_0).
    $$

    Let $\ka>\mu$ be regular. Using the stationarity of $\cal{M}$ (Definition \ref{def:sideaugment}) find a countable $M\prec H(\kappa)$ so that
    \begin{itemize}
        \item $M$ contains the parameters $\mu$, $\bb{A}$, $\cal{M}$, $(X,\tau)$, $\dot{E}$, and $\la a_0,p_0\ra$, and
        \item $\bar{M}:=M\cap H(\mu)$ is in $\cal{M}$.
    \end{itemize}
    Set $p_1:=p_0\cup\lb \bar{M}\rb$, and note that this is an $\in$-chain of elements of $\cal{M}$ since $\bar{M}\in\cal{M}$ and since $p_0\seq \bar{M}$. Since $\bb{A}$ is proper for $\cal{M}$, we may extend $a_0$ to an $(\bar{M},\bb{A})$-generic condition $a_1$. Since $\bb{A}\cap\bar{M}=\bb{A}\cap M$, $a_1$ is in fact $(M,\bb{A})$-generic. Note also that $\la a_1,p_1\ra$ is a condition in $\cal{M}(\bb{A})$ which extends $\la a_0,p_0\ra$. 
    
    We claim that $\la a_1,p_1\ra$ is the desired condition, and thus we need to define the name for the relevant countable subset of $\dot{E}$. Set
    $$
    M(\dot{E}):=\lb(\check{z},\la\bar{a},\bar{p}\ra)\in M:\la\bar{a},\bar{p}\ra\in\cal{M}(\bb{A})\we\la\bar{a},\bar{p}\ra\Vdash \check{z}\in\dot{E}\rb.
    $$
    Note that $\cal{M}(\bb{A})$ outright forces that $M(\dot{E})$ is a countable subset of $\dot{E}$. To complete the proof, we will show that 
    $$
    \la a_1,p_1\ra\Vdash \check{x}\in\operatorname{cl}(M(\dot{E})).
    $$
    As in the proof of the previous lemma, we will project $\dot{E}$ onto the 1st two coordinates and then apply the preservation of the countable tightness of $X$ by $\bb{A}$.

    
    For each $\in$-chain $s\supseteq p_0$ in $M$ of models from $\cal{M}$, define
    $$
    \dot{E}_s:=\lb\la\check{z},\bar{a}\ra:\exists \bar{p}\;\left(\la\bar{a},\bar{p}\ra\in\cal{M}(\bb{A})\we s\subseteq \bar{p}\we\la\bar{a},\bar{p}\ra\Vdash_{\cal{M}(\bb{A})} \check{z}\in\dot{E}\right)\rb.
    $$
    Note that $\dot{E}_s$ is a member of $M$, being definable in $H(\kappa)$ by parameters in $M$, and note as well that it is an $\bb{A}$-name.\\

    \textbf{Claim 1:} Let $s\in M$ with $s\supseteq p_0$, and let $a\leq a_1$ be a condition in $\bb{A}$ which is $(N,\bb{A})$-generic for all $N$ in $s$. Then $a$ forces that $x\in\operatorname{cl}(\dot{E}_s)$.\\

    \emph{Proof of Claim 1:} Let $G$ be an arbitrary $V$-generic filter over $\bb{A}$ with $a\in G$. Set $$
    E_s:=\dot{E}_s[G].
    $$
    Note that both $x$ and $E_s$ are in $M[G]$ and that $M[G]\prec H(\kappa)[G]$. It therefore suffices to show that $M[G]\models x\in\operatorname{cl}(E_s)$. Recalling that $\tau$ is a basis for the topology on $X$ in $V[G]$, fix $B\in M[G]\cap\tau$ with $x\in B$. Since $a_1\in G$ is $(M,\bb{A})$-generic, we know that $B\in\tau\cap M[G]=\tau\cap M$.

    Now we apply a density argument. Define $D_{B,s}$ to be the set of $\bar{a}\in\bb{A}$ so that one of the following two assertions holds:
    \begin{enumerate}
        \item no extension of $\bar{a}$ is generic for every $N\in s$, or
        \item there exist $\bar{p}\supseteq s$ and $z\in B$ so that $\la\bar{a},\bar{p}\ra\Vdash_{\cal{M}(\bb{A})}\check{z}\in\dot{E}$.
    \end{enumerate}
    Then $D_{B,s}$ is dense in $\bb{A}$ and is in $M$ as well. 
    Since $a_1\in G$ is $(M,\bb{A})$-generic, we may find a condition $\bar{a}\in G\cap M\cap D_{B,s}$. Since $a$ is in $G$ and is generic for every $N\in s$, the membership of $\bar{a}$ in $D_{B,s}$ cannot hold on account of option (1), and therefore must hold because option (2) obtains. By the elementarity of $M$, we can also find $\bar{p}\in M$ with $\bar{p}\supseteq s$ and $z\in B\cap M$ so that $\la\bar{a},\bar{p}\ra\Vdash \check{z}\in\dot{E}$. Then $z\in B\cap E_s$, since $\la\check{z},\bar{a}\ra\in\dot{E}_s$, with $\bar{p}$ as a witness. $\hfill\blacksquare\text{(Claim 1)}$\\

Note for the next claim that that $\dot{E}_s\cap M$ is an $\bb{A}$-name.\\

    \textbf{Claim 2:} Let $s\in M$ with $s\supseteq p_0$, and let $a\leq a_1$ be a condition in $\bb{A}$ which is $(N,\bb{A})$-generic for all $N$ in $s$. Then  $a$ forces in $\bb{A}$ that $x\in\operatorname{cl}(\dot{E}_s\cap M)$.\\

    \emph{Proof of Claim 2:} Recalling that $\tau$ is a basis for the topology on $X$ in the extension by $\bb{A}$ and that we only need to verify the relevant property with respect to basic open sets, fix an arbitrary $B\in\tau$ with $x\in B$. Also, fix a $V$-generic filter $G$ over $\bb{A}$ containing $a$. Since $M[G]\models x\in\operatorname{cl}(E_s)$, by the previous claim, and since $X$ is countably tight in $V[G]$, we may apply the elementarity of $M[G]$ to find a countable subset $F$ of $E_s$ in $M[G]$ so that $x\in\operatorname{cl}(F)$. Let $z\in B\cap F$, noting that $z\in M[G]$ because $F$, being countable and an element of $M[G]$, is a subset of $M[G]$. Since $a_1\in G$ is an $(M,\bb{A})$-generic condition, $z\in M[G]\cap X=M\cap X$. Moreover, because $z\in E_s\cap M[G]$, we can find $\bar{a}\in G\cap M[G]=G\cap M$ so that $\la\check{z},\bar{a}\ra\in\dot{E}_s$. Since $\la\check{z},\bar{a}\ra$ is in $M\cap \dot{E}_s$ and $\bar{a}\in G$, this finishes the proof of the claim. $\hfill\blacksquare\text{(Claim 2)}$\\

    Now we put all of this together. Let $U\in\tau$ be an arbitrary open set with $x\in U$, and let $\la a_2,p_2\ra\leq\la a_1,p_1\ra$ be arbitrary. We will find an extension of $\la a_2,p_2\ra$ that forces that $\check{U}\cap (\dot{E}_s\cap M)\neq\es$.

    By Claim 2, with $s:=p_2\cap M$, we can find an $a_3$ extending $a_2$, a specific $z\in U\cap M$, and a specific $\bar{a}\in\bb{A}$ with $a_3\leq \bar{a}$ so that $\la\check{z},\bar{a}\ra\in M\cap \dot{E}_s$. By the elementarity of $M$, we may find a $\bar{p}\in M$ with $p_2\cap M\subseteq \bar{p}$ so that $\la\bar{a},\bar{p}\ra$ is a condition in $\cal{M}(\bb{A})$ which forces $\check{z}\in\dot{E}$. Since $\la\bar{a},\bar{p}\ra$ is a condition in $\cal{M}(\bb{A})$, $\bar{a}$ is generic for every model in $\bar{p}$, and since $a_3\leq \bar{a}$, so is $a_3$. Moreover, as $\bar{p}\in M$ (and note that this implies $\bar{p}\in\bar{M}$) extends $p_2\cap M$, we know that $p_3:=\bar{p}\cup p_2$ is an $\in$-chain. Finally, as $a_3\leq a_2$ is generic for all models in $p_2$ and generic for all models in $\bar{p}$, we have that $\la a_3,p_3\ra$ is a condition in $\cal{M}(\bb{A})$ extending $\la\bar{a},\bar{p}\ra$ and $\la a_2,p_2\ra$.
    
    \end{proof}

\section{The Consistency of $\pfa_L(X)$}\label{Section:PFAL}

In this section we will prove the consistency of $\pfa_L(X)$. Recall the notation $Z^*:=Z\cap\operatorname{cof}(>\om)$ from Section \ref{Section:Poset}. We will use the symbol $\ps$ to abbreviate $\ps_L(X)$ in this section.

The bulk of the work is an inductive proof along $\be\in Z^*\cup\lim(Z^*)$ that if for all $\al\in Z^*\cap\be$, $\ps\upharpoonright\al$ preserves that $X$ is Lindel{\"o}f, then $\ps\upharpoonright\be$ preserves that $X$ is Lindel{\"o}f. There are three main cases and some subcases:
\begin{enumerate}
    \item Base case $\be=\min(Z^*)$;
    \item Limit case $\be\in\lim(Z^*)$;
    \item Successor case: $\be\in Z^*\bsl\lim(Z^*)$;
    \begin{enumerate}
        \item Subcase 1: $\sup(Z^*\cap\be)\in Z^*$, i.e., there is $\al\in Z^*$ so that $\be=\min(Z^*\bsl(\al+1))$;
        \item Subcase 2: $\sup(Z^*\cap\be)\notin Z^*$.
    \end{enumerate}
\end{enumerate}

In short: the base case is taken care of by \cite{GiltonHolshouser}, and Case 3(a) is handled by the side conditions augmentation machinery (this is true even if $J(\al)$ does not give the right kind of poset; we'll say more later). We'll need additional arguments for Cases (2) and 3(b). In the remainder of this section, we will address these other cases and then put everything together.

\subsection{The Remaining Cases}

\begin{proposition}\label{prop:limitcaselindelof}
    Let $\be\leq\theta$ be a limit point of $Z^*$, and suppose that for all $\al\in Z^*\cap\beta$, ${\bb P}\upharpoonright\alpha$ preserves that $X$ is Lindel{\"o}f. Then $\ps\upharpoonright\beta$ preserves that $X$ is Lindel{\"o}f. 
\end{proposition}
\begin{proof}
    
    Fix a condition $\la s_0,f_0\ra$ in $\bb{P}\upharpoonright\beta$ and a $(\bb{P}\upharpoonright\beta)$-name $\dot{\mathcal{U}}$ for an open cover of $X$. Our goal is to find an extension of $\la s_0,f_0\ra$ which forces in $\bb{P}\upharpoonright\beta$ that $\dot{\mathcal{U}}$ has a countable subcover.

    Let $\kappa>\theta$ be regular, and fix a countable $M^*\prec H(\kappa)$ so that $M^*$ has the following parameters: $J$, ${\bb P}$, $\mathcal{S}$, $\mathcal{T}$, $\la s_0,f_0\ra$, $(X,\tau)$, $\dot{\mathcal{U}}$, $Z$, $\theta$, and $\beta$. Set $M:=M^*\cap H(\theta)$, so that $M\in\mathcal{S}$. Since $s_0\in M$, Lemma \ref{lemma:tInM} implies that the closure of $s_0\cup\lb M\rb$ under intersections is a condition in the model sequence poset extending $s_0$. Let $s_1$ be this condition. Next, by Lemma \ref{lemma:addM}, let $f_1$ be a function so that $\la s_1,f_1\ra$ is a condition extending $\la s_0,f_0\ra$.
    
    We claim that $\la s_1,f_1\ra$ forces the desired conclusion. More precisely, we claim that $\la s_1,f_1\ra$ forces that the following is a countable subcover of $\dot{\mathcal{U}}$:
    $$
    M(\dot{\mathcal{U}}):=\lb(\check{B},\la s,f\ra):B\in M\cap\tau\we\la s,f\ra\in M\cap (\bb{P}\upharpoonright\beta)\we \la s,f\ra\Vdash_{\bb{P}\upharpoonright\beta}\check{B}\in\dot{\mathcal{U}}\rb.
    $$
    Because $M$ is countable,  $M(\dot{\mathcal{U}})$ is countable and hence forced to be a countable subset of $\dot{\mathcal{U}}$. Naturally, it takes some work to show that it's an open \emph{cover}.

    In order to show that $\la s_1,f_1\ra$ forces that $M(\dot{\mathcal{U}})$ is a cover, fix an arbitrary point $x\in X$ and an extension $\la s_2,f_2\ra$ of $\la s_1,f_1\ra$. The proof will proceed by finding a large enough transitive node, say $W$, in $M$. We will argue that we can add $W$ to $\la s_2,f_2\ra$ in a controlled way. Then, we will apply our induction hypothesis to the restriction of $\bb{P}\upharpoonright\beta$ to $W$. This will involve looking at the projection of the name $\dot{\mathcal{U}}$ and applying amalgamation arguments.

    Since $\beta\in\lim(Z^*)$ and  since $s_2\cap M$ is an element of $M\cap H(\beta)$, we may apply the elementarity of $M^*$ to find a node $W\in\mathcal{T}\cap M\cap H(\beta)$ so that $s_2\cap M\in W$. By Lemma \ref{lemma:addTopTransitive}, we may find an extension $s_3$ of $s_2$ in the model sequence poset so that $W\in s_3$ and so that the nodes of $s_3$ above $W$ are precisely the nodes of $s_2$ from $M$ onward. Moreover, we also know from Lemma \ref{lemma:addTopTransitive} that the small nodes of $s_3$ are either small nodes of $s_2$ or are of the form $N\cap W$, where $N\in s_2$ is small. This implies by Remark \ref{remark:condition5initialsegment} that $\la s_3,f_2\ra$ is a condition.
    
    We now build a $(\bb{P}\upharpoonright W)$-name to which we will apply our induction hypothesis. Namely, let $\dot{\mathcal{U}}_{W,s_3\cap M}$ be the $(\bb{P}\upharpoonright W)$-name consisting of all pairs $(\check{B},\la\bar{s},\bar{f}\ra)$ so that
    \begin{enumerate}
        \item $B\in\tau$ and $\la\bar{s},\bar{f}\ra\in \ps\upharpoonright W$, and
        \item there exists $\la s,f\ra\in \ps\upharpoonright\beta$ so that
        \begin{enumerate}
            \item $s_3\cap M\seq s$
            \item $\la s,f\ra\upharpoonright W=\la\bar{s},\bar{f}\ra$, and
            \item $\la s,f\ra\Vdash\check{B}\in\dot{\cal{U}}$.
        \end{enumerate}
    \end{enumerate}
    
    Note that in the definition of $\dot{\mathcal{U}}_{W,s_3\cap M}$, the symbol ``$\check{B}$" is being used equivocally, on the one hand for the canonical $(\bb{P}\upharpoonright W)$-name for $B$ and on the other hand for the canonical $({\bb P}\upharpoonright\beta)$-name for $B$. This will not cause any issues going forwards.
    
    Observe that the name $\dot{\mathcal{U}}_{W,s_3\cap M}$ is a member of $M^*$, being definable in $H(\kappa)$ from parameters in $M^*$.\\

    \textbf{Claim 1:} $\la s_3,f_2\ra\upharpoonright W$ forces in $\bb{P}\upharpoonright W$ that $\dot{\mathcal{U}}_{W,s_3\cap M}$ is an open cover of $X$.\\

    \emph{Proof of Claim 1:} Fix a point $z\in X$ and an extension $\la\bar{t},\bar{g}\ra$ of $\la s_3,f_2\ra\upharpoonright W$ in $\ps\cap W$. By Lemma \ref{lemma:transitivestronglygeneric}, we can find a condition $\la t',g'\ra\in\ps\upharpoonright\beta$ extending $\la s_3,f_2\ra$ and $\la\bar{t},\bar{g}\ra$. By further extension if necessary we may assume that for a specific $B\in\tau$ with $z\in B$, $\la t',g'\ra$ forces that $\check{B}\in\dot{\cal{U}}$. Then the pair $(\check{B},\la t',g'\ra\upharpoonright W)$ is in $\dot{\cal{U}}_{W,s_3\cap M}$. Hence $\la t',g'\ra\upharpoonright W$, which extends $\la s_3,f_2\ra\upharpoonright W$, forces that $\check{B}\in\dot{\cal{U}}_{W,s_3\cap M}$. $\hfill\blacksquare\text{(Claim 1)}$\\

    \textbf{Claim 2:} $\la s_3,f_2\ra\upharpoonright W$ forces that $M\cap \dot{\mathcal{U}}_{W,s_3\cap M}$ is an open cover of $X$.\footnote{Note that we do literally mean $M\cap \dot{\mathcal{U}}_{W,s_3\cap M}$, i.e., all pairs from the name which are also members of $M$.}\\

    \emph{Proof of Claim 2:} Fix $z\in X$ and an extension $\la\bar{t},\bar{g}\ra$ of $\la s_3,f_2\ra\upharpoonright W$ in $\ps\upharpoonright W$. Let $G_W$ be a $V$-generic filter over $\ps\upharpoonright W$ with $\la\bar{t},\bar{g}\ra\in G_W$.
    
    We'd like to apply Lemma \ref{lemma:localproper} to conclude that we have 
    $$
    M[G_W]\cap V=M.
    $$
    Indeed, let the ``$\alpha$" and ``$\beta$" of the statement of that lemma equal $\zeta$, where $W=H(\zeta)$; let ``$\ga$" from that lemma be $\theta$; and let ``$M^*$" from that lemma be equal to $M$. Note that $M\cap W$ is a member of $s_3\cap W$ by the closure of $s_3$ under intersections. Lemma \ref{lemma:localproper} implies that $\la s_3,f_2\ra\upharpoonright W$ is a generic condition for $M$ in $\ps\upharpoonright W$.

    Continuing, we may apply the previous claim to conclude that $$
    \cal{U}_W:=\dot{\cal{U}}_{W,s_3\cap M}[G_W]
    $$ 
    is an open cover of $X$. Moreover, it is in $M[G_W]$. Accordingly, we apply the fact that $X$ is Lindel{\"o}f in $V[G_W]$ and the elementarity of $M[G_W]$ to find a countable subcover $\cal{B}_W$ of $\cal{U}_W$ with $\cal{B}_W\in M[G_W]$. Fix $B\in\cal{B}_W$ with $z\in B$. Note that because $\cal{B}_W\seq M[G_W]$, we know $B\in M[G_W]\cap\tau=M\cap\tau$, where the equality holds because $\tau\subset V$. Applying the elementarity of $M[G_W]$, let $\la\bar{s},\bar{f}\ra$ be a condition in $G_W\cap M[G_W]$ so that $(\check{B},\la\bar{s},\bar{f}\ra)$ is a member of $\dot{\cal{U}}_{W,s_3\cap M}$. Again by properness, we know that $\la\bar{s},\bar{f}\ra$ is in fact in $M$, not just $M[G_W]$. Thus $(\check{B},\la\bar{s},\bar{f}\ra)$ is a member of $M\cap \dot{\cal{U}}_{W,s_3\cap M}$. Since $\la\bar{s},\bar{f}\ra$ is in $G_W$, it is compatible with $\la \bar{t},\bar{g}\ra$, and this suffices to finish the proof of the claim. $\hfill\blacksquare\text{(Claim 2)}$\\

    Now we put everything together, recalling the arbitrary point $x$ fixed in the fourth paragraph of the proof. Using the second claim, we may find 
    \begin{itemize}
        \item an extension $\la\bar{s}_4,\bar{f}_4\ra$ in $\ps\upharpoonright W$ of $\la s_3,f_2\ra\upharpoonright W$;
        \item some $B\in\tau\cap M$ with $x\in B$;
        \item a condition $\la\bar{s},\bar{f}\ra\in M\cap W$ with $\bar{s}\supseteq s_3\cap M$ and with $\la\bar{s}_4,\bar{f}_4\ra$ extending $\la\bar{s},\bar{f}\ra$ in $\bb{P}\upharpoonright W$;
        \item and a condition $\la s,f\ra\in M$ so that 
        $$
        \la s,f\ra\upharpoonright W=\la\bar{s},\bar{f}\ra\text{ and }\la s,f\ra\Vdash\check{B}\in\dot{\mathcal{U}}.
        $$
    \end{itemize}

    We will prove the following claim, which will finish the proof of the proposition.\\

    \textbf{Claim 3}: there is a condition $\la s_4,f_4\ra$ which extends all three of the conditions $\la\bar{s}_4,\bar{f}_4\ra$, $\la s_3,f_2\ra$ and $\la s,f\ra$.\\

    \emph{Proof of Claim 3:} To begin, let $\De_0,\dots,\De_{\delta-1}$ enumerate in $\in$-increasing order the transitive nodes of $s$ above $W$. Note that these are in $M$ since $s\in M$. Because we (most likely) have these new transitive nodes, we will need to add some more models in order to ensure closure under intersections. This will form a new condition $s_4$ in the model sequence poset. After we do this, we will construct a function $f_4$ so that $\la s_4,f_4\ra$ is the desired condition.

    Let $W'$ denote the least transitive node in $s_3$ above $M$ if there is one and $H(\theta)$ otherwise. Define $E_M$ to be the interval $[M,W')$ of small nodes of $s_3$, starting from $M$ and below $W'$, and let  $M=M_0\in\dots\in M_{l-1}$ be the $\in$-increasing enumeration  of $E_M$. For each $\ga<\de$, let 
    $$
    F_\ga:=\lb M_j\cap\Delta_\ga:j<l\rb,
    $$
    an increasing sequence of small nodes. Form a precondition $s'_4$ in the model sequence poset by taking $\bar{s}_4$ followed by the tail segment of $s$ starting with $W$, followed by the tail segment of $s_3$ starting from $M$, that is to say
    $$
    s'_4:=\bar{s}_4\cup(s\bsl W)\cup (s_3\bsl M).
    $$
    Observe that for each $\ga<\de$, the node of $s'_4$ which is directly below $\De_\ga$ is a node of $s$. Observe as well that $s'_4$ contains all three of the conditions $\bar{s}_4$, $s$, and $s_3$. Let $s_4$ be the sequence obtained by adding, for each $\ga<\de$, the nodes of $F_\ga$ to $s'_4$, directly below $\De_\ga$.

    We now argue that $s_4$ is a condition in the model sequence poset. We first show that $s_4$ is $\in$-increasing. Since, for each $\ga<\de$, $F_\ga$ is $\in$-increasing and since $s'_4$ is $\in$-increasing, it suffices to verify that the new set is $\in$-increasing at the boundaries of the new $F_\ga$ intervals. Fix $\ga<\de$. On the one hand, the top node of $F_\ga$ is $M_{l-1}\cap\De_\ga$, which is in $\De_\ga$. On the other hand, the bottom node of $F_\ga$ is $M\cap\De_\ga$. However, if $K$ is the top node of the above precondition directly below $\De_\ga$, then $K$ is in fact the top node of $s$ below $\De_\ga$. Consequently, $K\in M$ because $s\in M$ is finite. Hence $K\in M\cap\De_\ga$.

    Our second claim is that $s_4$ is closed under intersections. By Lemma \ref{lemma:weakclosure}, it suffices to show that if $W^*$ and $K$ are nodes in $s_4$ of transitive and small type respectively, and if $W^*\in K$, then $W^*\cap K\in s_4$. We have a number of cases based upon where $W^*$ lives.

    Case 1: $W^*\in\bar{s}_4$. Since $\bar{s}_4\seq W$, this implies that $W^*$ is a member of $W$ and hence that $W^*\seq W$. We take subcases based on where $K$ lives:
    \begin{itemize}
        \item If $K\in\bar{s}_4$, then we're done because $\bar{s}_4$ is closed under intersections. 
        \item If $K\in s_3$, then because $W\in s_3$ as well, we have that $K\cap W\in s_3\cap W\seq\bar{s}_4$. Since $W^*\in\bar{s}_4$, we conclude that $K\cap W^*=(K\cap W)\cap W^*\in\bar{s}_4$.
        
        \item If $K\in s$, then $K\cap W\in s\cap W=\bar{s}\seq\bar{s}_4$. Then continue as in the previous subcase.
        \item The final subcase is that $K$ is a new small node, i.e., of the form $K=M_j\cap\De_\ga$ for some $j<l$ and $\ga<\de$. Then $K\cap W=(M_j\cap\De_\ga)\cap W=M_j\cap W$, where the second equality holds because $W\in\De_\ga$ and $\De_\ga$ is transitive. Since $M_j$ and $W$ are both in $s_3$, we conclude $K\cap W\in s_3\cap W\seq\bar{s}_4$. Then $K\cap W^*=(K\cap W)\cap W^*\in\bar{s}_4$ as well, completing the proof in this subcase.
    \end{itemize}

Case 2: $W^*=W$. Since $W\in K$, the node $K$ must occur in $s$, or in $s_3$, or be one of the new nodes added. We take subcases based upon these possibilities. 
\begin{itemize}
    \item Suppose $K\in s$. Since $W\in s_3\cap M\seq s$, we know that $W=W^*$ and $K$ are both in $s$, so their intersection is in $s\seq s_4$.
    \item If $K\in s_3$, then as $W^*=W$ is in $s_3$ as well, we conclude that $W\cap K\in s_3\seq s_4$.
    \item If $K$ is a new small node, then $K=M_j\cap\De_\ga$ for some $j<l$ and $\ga<\de$. Then $K\cap W=(M_j\cap\De_\ga)\cap W=M_j\cap W$. As $M_j\cap W$ is in $s_3$, this finishes the proof in this subcase.
\end{itemize}

Case 3: $W^*=\De_\ga$ for some $\ga<\de$. As in Case 2, $K$ must then be a member of $s$, or a member of $s_3$, or a new small node.
\begin{itemize}
    \item $K\in s$. Then as $\De_\ga\in s$, apply the closure of $s$ under intersections.
    \item $K$ occurs at or above $M$ in $s_3$. Then $K\cap W'\in E_M$, and so $K\cap\De_\ga=(K\cap W')\cap\De_\ga\in F_\ga\seq s_4$.
    \item $K$ is a new small node, say $K=M_j\cap\De_{\ga^*}$ for some $\ga^*<\de$. Since $\De_\ga\in K\seq\De_{\ga^*}$ and since $\De_{\ga^*}$ is transitive, we know $\De_\ga\seq\De_{\ga^*}$. So $K\cap\De_\ga=(M_j\cap\De_{\ga^*})\cap\De_\ga=M_j\cap\De_\ga\in F_\ga\seq s_4$.
\end{itemize}

Case 4: $W^*$ occurs in $s_3$ above $M$. Then $K$ also occurs in $s_3$ since $K$ is above $W^*$, and we are done by applying the closure of $s_3$ under intersections.\\

This completes the construction of the condition $s_4$ in the model sequence poset. Now we construct the function $f_4$. The domain of $f_4$ will be the union of the domains of $f_2$, $f$, and $\bar{f}_4$. $f_4$ below $W$ will equal $\bar{f}_4$, and $f_4$ above $M$ will equal $f_2$. The remaining transitive nodes of $s_4$, in order, are $W$ followed by the $\De_\ga$. Since these are all elements of $s$ (recall $W\in s_3\cap M\seq s$), they are potentially elements of the domain of $f$ (recall that condition (2) of Definition \ref{def:PFAiteration} does not demand that the domain of the working part include all of the transitive nodes present in the first coordinate). For the nodes among these that are in the domain of $f$, we will need to strengthen the corresponding condition in the range of $f$.

First we deal with $W$. If $W\notin\dom(f)$, then $W$ will not be in $\dom(f_4)$. On the other hand, if $W\in\dom(f)$, then $f(W)$ is forced by $\la s,f\ra\upharpoonright W=\la\bar{s},\bar{f}\ra$ to be a generic condition for the small nodes of $s$ between $W$ and $\De_0$ (the next transitive node of $s$). Let $K_0$ be the top node of $s$ directly below $\De_0$, and for non-triviality, we suppose that $K_0$ is small. Then we added $M_0\cap\De_0,\dots,M_{l-1}\cap\De_0$ above $K_0$ and below $\De_0$. Since $\la s,f\ra$ is a member of $M$, so is the $(\ps\upharpoonright W)$-name $f(W)$. Thus $f(W)\in M\cap\De_0=M_0\cap\De_0$, the node of $s_4$ directly above $K_0$. By Lemma \ref{lemma:genericforincreasingsequence}, we may find a $(\ps\upharpoonright W)$-name $f_4(W)$ extending $f(W)$ which is forced to be generic for all of the $M_j\cap\De_0$. This completes the definition of $f_4(W)$.

Now we deal with an arbitrary $\De_\ga$ in $\dom(f)$. First suppose that $\De_{\ga+1}$ exists (equivalently, that $\ga+1<\de$). Let $K_\ga$ be the top node of $s$ below $\De_{\ga+1}$, where we assume for non-triviality that $K_\ga$ is small. Then $\la s,f\ra\upharpoonright\De_\ga$ forces that $f(\De_\ga)$ is generic for all of the small nodes of $s$ in the interval $(\De_\ga,K_\ga]$. As in the previous paragraph, $f(\De_\ga)\in M\cap\De_{\ga+1}$, which is the smallest node of $s_4$ directly above $K_\ga$. We again apply Lemma \ref{lemma:genericforincreasingsequence} to find a name $f_4(\De_\ga)$ extending $f(\De_\ga)$ which is generic for each of the $M_j\cap\De_{\ga+1}$. On the other hand, if $\De_{\ga+1}$ does not exist, i.e., if $\ga=\de-1$, then $\De_\ga$ is the topmost transitive node of $s$. Let $K_\ga$ be the topmost node of $s$, where we again assume that $K_\ga$ is small for the sake of non-triviality. Note that $M$ is the next node of $s_4$. Since $f(\De_\ga)$ is a member of $M$, we yet again apply Lemma \ref{lemma:genericforincreasingsequence} to find an extension $f_4(\De_\ga)$ of $f(\Delta_\gamma)$ which is generic for each of the models (all of which are small) in the interval $[M,W')$ of $s_4$ (equivalently, of $s_3$).

This completes the definition of the pair $\la s_4,f_4\ra$. By Lemma \ref{lemma:genericforInterval}, it is a condition. $\hfill\blacksquare\text{(Claim 3)}$\\

Now we know that $\la s_4,f_4\ra$ is a condition extending $\la\bar{s}_4,\bar{f}_4\ra$, $\la s_3,f_2\ra$, and $\la s,f\ra$. Since $\la s_4,f_4\ra$ extends $\la s,f\ra$, it forces that $\check{B}\in\dot{\cal{U}}$. The condition $\la s_4,f_4\ra$ also forces that $\check{B}\in M(\dot{\cal{U}})$, with $\la s,f\ra$ as a witness. Since $x$ was an arbitrary point in $X$ and since $\la s_4,f_4\ra$ extends $\la s_3,f_2\ra$ which in turn extends $\la s_1,f_1\ra$, this completes the proof that $\la s_1,f_1\ra$ forces that $M(\dot{\cal{U}})$ is a cover of $X$. This completes the proof of the proposition.
\end{proof}

The following proposition deals with successor cases that are not covered by the machinery from Sections \ref{Section:EmbeddingLemmas} and \ref{Section:SCAPreservation}. The proof is quite similar to the proof of the previous proposition.

\begin{proposition}\label{prop:weirdsuccessorcaseLindelof}
    Suppose that $\be\in Z^*\bsl\lim(Z^*)$ and that $\al:=\sup(Z^*\cap\be)\notin Z^*$. Suppose further that for all $\bar{\al}<\al$ with $\bar{\al}\in Z^*$, $\ps\upharpoonright\bar{\alpha}$ preserves that $X$ is Lindel{\"o}f. Then $\ps\upharpoonright\beta$ preserves that $X$ is Lindel{\"o}f.
\end{proposition}
\begin{proof}
    Fix a $(\ps\upharpoonright\beta)$-name $\dot{\cal{U}}$ for an open cover of $X$ and a condition $\la s_0,f_0\ra$ in $\ps\upharpoonright\beta$. We will find an extension of $\la s_0,f_0\ra$ and a name for a countable subset of $\dot{\cal{U}}$ forced by that extension to be a cover of $X$.

    Let $\ka>\theta$ be regular and $M^*\prec H(\kappa)$ a countable elementary submodel containing the following parameters: $\ps$, $J$, $\cal{S}$, $\cal{T}$, $(X,\tau)$, $\la s_0,f_0\ra$, $\be$, $\dot{\cal{U}}$, $Z^*$, and $\theta$, and define $M:=M^*\cap H(\be)$, noting that $M\in\cal{S}$. Apply Lemma \ref{lemma:addM} to find an extension $\la s_1,f_1\ra$ of $\la s_0,f_0\ra$ with $M\in s_1$. We claim that $\la s_1,f_1\ra$ forces that the following is a countable subcover of $\dot{\cal{U}}$:
    $$
    M(\dot{\cal{U}}):=\lb(\check{B},\la t,g\ra):B\in M\cap\tau\we\la t,g\ra\in M\cap (\ps\upharpoonright\beta)\we\la t,g\ra\Vdash\check{B}\in\dot{\cal{U}}\rb.
    $$
    We show that no extension of $\la s_1,f_1\ra$ can force a specific counterexample. Towards this end, fix an extension $\la s_2,f_2\ra$ of $\la s_1,f_1\ra$ and a specific $x\in X$, and we find a further extension of $\la s_2,f_2\ra$ which forces $x\in\bigcup M(\dot{\cal{U}})$.

    Our first step will be to extend $\la s_2,f_2\ra$ to a condition $\la s_3,f_3\ra$ whose top node below $H(\al)$ is transitive, and so that this transitive node is not in $\dom(f_3)$. While not absolutely necessary, finding such a condition will smooth out the details later. To begin, if $\la s_2,f_2\ra$ already has this property, i.e., if $\max(s_2\cap H(\al))$ is transitive and is not in $\dom(f_2)$, then we set $\la s_3,f_3\ra=\la s_2,f_2\ra$. So suppose otherwise. Observe that since $M\in s_2\setminus H(\al)$, we may define $M_0:=\min(s_2\setminus H(\al))$, which may be $M$ itself. We claim that $M_0$ is a small node. Otherwise, $M_0=H(\bar{\de})$ for some $\bar{\de}\in Z^*$. Since $M_0\notin H(\al)$, $\bar{\de}\geq\al$. But $\bar{\de}\neq\al$, since $\al\notin Z^*$. Thus $\bar{\de}\geq\be$, because $\be=\min(Z^*\bsl\al)$. This contradicts $s_2\seq H(\be)$.

    Now we know that $M_0$ is a small node. Let $\de_0$ be the least element of $Z^*\cap\al$ so that $s_2\cap H(\al)\in H(\de_0)$. Note that $M_0\not\seq H(\de_0)$, since if it were, that would imply $M_0\in H(\al)$. Then by Lemma \ref{lemma:addTopTransitive}, we may define $\de_1:=\min(M_0\setminus\de_0)$, so that $\de_1\in Z^*$, and we may find a condition $s_3$ in the model sequence poset which extends $s_2$ and which has the following properties:
    \begin{itemize}
        \item $s_3$ contains $H(\de_1)$;
        \item all new nodes in $s_3$ are either transitive or of the form $N\cap W$ for $N\in s_2$ and $W\in\cal{T}$;
        \item the nodes of $s_3$ above $H(\de_1)$ are exactly the nodes of $s_2\setminus H(\delta_1)$.
    \end{itemize}

    Note that $\la s_3,f_2\ra$ is a condition by Remark \ref{remark:condition5initialsegment} and by the second bullet point above. For uniformity of notation, we set $f_3:=f_2$, so that we've now obtained the condition $\la s_3,f_3\ra$.

    One more bit of notation and an observation: set $Q$ to be the top transitive node of $s_3$. We note that $Q\in M$. Indeed, let $M_0,\dots,M_k=M$ enumerate in $\in$-increasing order the nodes of $s_2$ (equivalently, in this case, $s_3$) from $M_0$ to $M$. Each of these nodes is small, and hence $M_0\in M$. This and the countability of $M_0$ imply that $M_0\seq M$. Since $Q\in M_0$, we conclude that $Q\in M$.\\

    Let $\dot{\cal{U}}_{Q,s_3\cap M}$ be the $(\ps\upharpoonright Q)$-name consisting of pairs $(\check{B},\la\bar{t},\bar{g}\ra)$ so that
    \begin{enumerate}
        \item $B\in\tau$ and $\la\bar{t},\bar{g}\ra\in\ps\upharpoonright Q$ and
        \item there exists a condition $\la t,g\ra\in\ps\upharpoonright\beta$ so that
        \begin{enumerate}
            \item $\la t,g\ra\upharpoonright Q=\la\bar{t},\bar{g}\ra$
            \item $t$ extends $s_3\cap M$ in the model sequence poset
            \item the condition $\la t,g\ra$ forces that $\check{B}\in\dot{\cal{U}}$.
        \end{enumerate}
    \end{enumerate}

    \textbf{Claim 1:} 
    $\la s_3\cap Q,f_3\ra$ (which is a condition in $\ps\upharpoonright Q$) forces that $\dot{\cal{U}}_{Q,s_3\cap M}$ is an open cover of $X$.\\

    \emph{Proof of Claim 1:} Fix a point $z\in X$ and an extension $\la\bar{t},\bar{g}\ra$ of $\la s_3\cap Q,f_3\ra$ in $\ps\upharpoonright Q$. We find a further extension that forces $z$ to be in $\bigcup\dot{\cal{U}}_{Q,s_3\cap M}$.

    Since $\la\bar{t},\bar{g}\ra$ extends $\la s_3\cap Q,f_3\ra=\la s_3,f_3\ra\upharpoonright Q$, Lemma \ref{lemma:transitivestronglygeneric} implies that $\la\bar{t},\bar{g}\ra$ and $\la s_3,f_3\ra$ are compatible in $\ps\upharpoonright\beta$. Accordingly, let $\la t',g'\ra$ be a condition in $\ps\upharpoonright\beta$ extending both. By extending further if necessary, we can find a specific $B\in\tau$ with $z\in B$ so that $\la t',g'\ra$ forces that $\check{B}\in\dot{\cal{U}}$. Then the condition $\la t',g'\ra\upharpoonright Q$, which extends $\la\bar{t},\bar{g}\ra$,  forces that $\check{B}\in\dot{\cal{U}}_{Q,s_3\cap M}$, since $(\check{B},\la t',g'\ra\upharpoonright Q)$ is a member of $\dot{\cal{U}}_{Q,s_3\cap M}$. Since $z$ and $\la\bar{t},\bar{g}\ra$ were arbitrary, this completes the proof of the claim.
    \hfill$\blacksquare$(Claim 1)\\

    We now make a familiar move by considering the $(\ps\upharpoonright Q)$-name $M\cap\dot{\cal{U}}_{Q,s_3\cap M}$.\\

    \textbf{Claim 2:} The condition $\la s_3\cap Q,f_3\ra$ forces that $M\cap\dot{\cal{U}}_{Q,s_3\cap M}$ is an open cover of $X$.\\

    \emph{Proof of Claim 2:} Let $G_Q$ be a $V$-generic filter over $\ps\upharpoonright Q$ containing $\la s_3\cap Q,f_3\ra$, and let $y\in X$ be an arbitrary point. Then $M[G_Q]$ satisfies that $\cal{U}_{Q,s_3\cap M}$ is an open cover of $X$. Since $X$ is Lindel{\"o}f in $V[G_Q]$, we may apply the elementarity of $M[G_Q]$ to find a countable subcover $\cal{B}$ with $\cal{B}\in M[G_Q]$. Fix $B\in\cal{B}$ with $y\in B$. By the elementarity of $M[G_Q]$, let $\la\bar{t},\bar{g}\ra\in G_Q\cap M[G_Q]$ be a condition so that $(\check{B},\la\bar{t},\bar{g}\ra)\in\dot{\cal{U}}_{Q,s_3\cap M}$. By the properness guaranteed by Lemma \ref{lemma:localproper}, $\la\bar{t},\bar{g}\ra$ and $B$ are both in $M$, and hence $(\check{B},\la\bar{t},\bar{g}\ra)$ is in $M\cap\dot{\cal{U}}_{Q,s_3\cap M}$. This completes the proof. \hfill $\blacksquare$ (Claim 2).\\

    In light of the previous claim, we can find some $(\check{B},\la\bar{t},\bar{g}\ra)\in M\cap\dot{\cal{U}}_{Q,s_3\cap M}$ so that $\la s_3\cap Q,f_3\ra$ is compatible with $\la\bar{t},\bar{g}\ra$ and so that $B$ contains $x$ (the specific point fixed towards the beginning of the proof). Since  $(\check{B},\la\bar{t},\bar{g}\ra)$ is in $M$, we can find $t$ and $g$ in $M$ which witness that $(\check{B},\la\bar{t},\bar{g}\ra)$ is in $\dot{\cal{U}}_{Q,s_3\cap M}$.

    To finish the proof, we will show that $\la s_3,f_3\ra$ and $\la t,g\ra$ are compatible. Let $\la\bar{s}^*,\bar{f}^*\ra$ be a condition in $\ps\upharpoonright Q$ which extends $\la s_3\cap Q,f_3\ra$ and $\la\bar{t},\bar{g}\ra$. Let $s_4'$ be the sequence $\bar{s}^*$ followed by all of the nodes of $t$ starting from $Q$, then followed by all of the nodes of $s_3$ starting from $M$. In other words,
    $$
    s'_4:=\bar{s}^*\cup (t\bsl Q)\cup (s_3\bsl M).
    $$
    Then $s'_4$ is $\in$-increasing (i.e., a precondition) since all of the nodes of $\bar{s}^*$ are in $Q$ and since all of the nodes of $t$, and in particular those of $t\bsl Q$, are in $M$. Observe that both $s_3$ and $t$ are subsets of $s'_4$: $t$ is a subset because $t\cap Q=\bar{t}$ and $\bar{s}^*$ extends $\bar{t}$. On the other hand, $s_3$ is a subset of $s'_4$ because $s_3=(s_3\cap Q)\cup\lb Q\rb\cup (s_3\bsl M)$, and also since $s_3\cap Q\seq\bar{s}^*$ and $Q\in t$.

    We next add nodes to close $s'_4$ under intersections. Let $M=M_k\in\dots\in M_{k^*}$ list the nodes of $s_3$ starting from $M$ onward (these are also the nodes of $s_2$ from $M$ onward). Each of these nodes is a small node. Furthermore, let $Q=Q_0,\dots,Q_l$ list the transitive nodes of $t$ starting from $Q$ in $\in$-increasing order. For each $j$ with $1\leq j\leq l$ (i.e., we're ignoring $Q_0$), we add the sequence $M_k\cap Q_j,\dots,M_{k^*}\cap Q_j$ to $s'_4$ directly below $Q_j$. Let $s_4$ denote this sequence of nodes. Then $s_4$ is also $\in$-increasing. Indeed, for each $j$ with $1\leq j\leq l$, the max node of $s'_4$ below $Q_j$ is a node of $t$, and this node is in $M\cap Q_j$. However, $M\cap Q_j$ is the bottom node of the sequence $\la M_i\cap Q_j:k\leq i\leq k^*\ra$ that we added to $s'_4$. By using Lemma \ref{lemma:weakclosure}, one can check that $s_4$ is closed under intersections.

    We finish the proof of this proposition by defining a working part $f_4$. We will do so by taking $\bar{f}^*$ below $Q$ and by extending $g$ at and above $Q$. To wit, let $Q'_0,\dots,Q'_m$ enumerate the nodes of $\dom(g)\setminus Q$; note that perhaps $Q=Q'_0$. Since $\la t,g\ra$ is a condition, for each $i\leq m$ we know that $\la t,g\ra\upharpoonright Q'_i$ forces that $g(Q'_i)$ is a generic condition in $J(Q'_i)$ for each $N\in t$ with $Q'_i\in N$. Since $\la t,g\ra$ is a member of $M$, we may apply Lemma \ref{lemma:genericforincreasingsequence} to each $i\leq m$ and find a $(\ps\upharpoonright Q'_i)$-name $\dot{u}_i$ so that $\la t,g\ra\upharpoonright Q'_i$ forces $\dot{u}_i\leq_{J(Q'_i)}g(Q'_i)$ and forces that $\dot{u}_i$ is a generic condition for $M_j$ for each $k\leq j\leq k^*$ (recall that $M_k=M$).

    Let $f_4$ be the function $\bar{f}^*\cup\lb Q'_i\mapsto\dot{u}_i:i\leq m\rb$. We now see that $\la s_4,f_4\ra$ is a condition, using Lemma \ref{lemma:genericforInterval} to secure part (5) of Definition \ref{def:PFAiteration}.

\end{proof}

\subsection{Putting it All Together}

We are now ready to prove our first main theorem:

\begin{theorem}\label{theorem:Lindelof}
    Under Assumption \ref{GroundModelassumption}, suppose that $X$ is a Lindel{\"o}f space. Then the poset $\ps_L(X)$, where ``$L$" denotes ``Lindel{\"o}f", forces that $\pfa_L(X)$ holds.
\end{theorem}
\begin{proof}

    We first note, by Lemma \ref{lemma:cardinalstructure}, that $\om_1$ and all $\lam\geq\theta$ are preserved and that $\theta$ becomes $\aleph_2$ in the extension.

    The main thing that we need to show is that $X$ remains Lindel{\"o}f after forcing with $\ps:=\ps_L(X)$. We prove by induction on $\be\in Z^*\cup\lim(Z^*)$ that $\ps\upharpoonright\beta$ preserves that $X$ is Lindel{\"of}; note that we are including $\theta$ itself in $\lim(Z^*)$.

    If $\be=\min(Z^*)$, then $\ps\upharpoonright\be$ is strongly proper, and hence by the results from \cite{GiltonHolshouser}, it preserves that $X$ is Lindel{\"o}f.

    Next, suppose that $\be\in\lim(Z^*)$ and that the result holds for all $\al$ below $\be$. By Proposition \ref{prop:limitcaselindelof}, $\ps\upharpoonright\beta$ preserves that $X$ is Lindel{\"o}f.

    Now we consider the successor cases. Fix $\be\in Z^*\setminus\lim(Z^*)$, and let $\al:=\sup(Z^*\cap\beta)$. Suppose first that $\al\in Z^*$. Recalling Notation \ref{notation:weirdJ}, Proposition \ref{prop:denseembedding} implies that there is a dense embedding from $D_{\al,\be}$ to $(\ps\upharpoonright\al)\ast(\dot{\cal{M}}_\al(J'(\al)))$. By induction, $\ps\upharpoonright\alpha$ preserves that $X$ is Lindel{\"o}f. There are two options for the value of $J'(\al)$: on the one hand, if $J'(\al)$ outputs the trivial poset, then $\ps\upharpoonright\al$ forces that $\dot{\cal{M}}_\al(J'(\al))$ is strongly proper (since it's isomorphic to forcing with finite $\in$-chains of countable models), and so the results from \cite{GiltonHolshouser} imply that it preserves that $X$ is Lindel{\"o}f. On the other hand, if $J'(\al)=J(\al)$ (i.e., if $J(\al)$ is forced to be the appropriate kind of poset) then by Lemma \ref{lemma:SCAlindelof}, we conclude that $\ps\upharpoonright\al$ forces that $\dot{\cal{M}}_\al(J(\al))$ preserves that $X$ is Lindel{\"o}f. In either case, on account of the dense embedding, we conclude that $\ps\upharpoonright\beta$ preserves that $X$ is Lindel{\"o}f.

    Second, suppose that $\al\notin Z^*$. By Proposition \ref{prop:weirdsuccessorcaseLindelof} and by our induction hypothesis, we conclude that $\ps\upharpoonright\beta$ preserves that $X$ is Lindel{\"o}f.

    Since the case $\beta=\theta$ was already taken care of as part of the limit case, we conclude that $\ps$ preserves that $X$ is Lindel{\"o}f. Since $\ps$ also preserves $\omega_1$ and forces that $\theta$ becomes $\omega_2$, we may now apply the usual arguments to conclude that $\ps$ forces $\pfa_L(X)$.
\end{proof}

\section{The Consistency of $\pfa_{CT}(X)$}\label{Section:PFACT}

In this section, we use the same case division as in the beginning of Section \ref{Section:PFAL}. We also let $\ps$ abbreviate $\ps_{CT}(X)$ where $X$ is a countably tight space of size $<\theta$.

\subsection{The Remaining Cases:}

\begin{proposition}\label{prop:limitcaseCT}
     Let $\be\in\lim(Z^*)$, and suppose that for all $\al\in Z^*\cap\be$, $\ps\upharpoonright\alpha$ preserves that $X$ is countably tight. Then $\ps\upharpoonright\beta$ preserves that $X$ is countably tight.
\end{proposition}
\begin{proof}
    Fix a point $x\in X$ and a $(\ps\upharpoonright\be)$-name $\dot{E}$ for a subset of $X$ so that
    $$
    \Vdash_{\ps\upharpoonright\be}\check{x}\in\operatorname{cl}(\dot{E}).
    $$
    Let $\la s_0,f_0\ra$ be an arbitrary condition in $\ps\upharpoonright\be$.  We will find an extension $\la s_1,f_1\ra$ of $\la s_0,f_0\ra$ and a name $\dot{E}_0$ for a countable subset of $\dot{E}$ so that
    $$
    \la s_1,f_1\ra\Vdash_{\ps\upharpoonright\beta}\check{x}\in\operatorname{cl}(\dot{E}_0).
    $$

    Let $\ka>\theta$ be regular and $M^*\prec H(\kappa)$ countable so that $M^*$ contains the following parameters: $J$, $\ps$, $\cal{S}$, $\cal{T}$, $\la s_0,f_0\ra$, $(X,\tau)$, $\theta$, $\beta$, $\dot{E}$, and $x$. Define $M:=M^*\cap H(\beta)$ so that $M\in\mathcal{S}$. As in the proof of Proposition \ref{prop:limitcaselindelof}, find an extension $\la s_1,f_1\ra$ of $\la s_0,f_0\ra$ with $M\in s_0$. Now define
    $$
    M(\dot{E}):=\lb (\check{z},\la t,g\ra)\in M:z\in X\we \la t,g\ra\Vdash_{\ps\upharpoonright\beta}\check{z}\in\dot{E}\rb.
    $$
    We claim that 
    $$
    \la s_1,f_1\ra\Vdash\check{x}\in\operatorname{cl}(M(\dot{E})).
    $$

    Recall that it suffices to work with basic open sets to verify that a point is in the closure of a set, and also recall that $\tau$ is a basis for the topology on $X$ in the extension. In light of this, fix an arbitrary extension $\la s_2,f_2\ra$ of $\la s_1,f_1\ra$ and an open set $U\in\tau$ with $x\in U$, and we will find a further extension of $\la s_2,f_2\ra$ which forces that $\check{U}\cap\dot{E}\neq\es$.

    As in the proof of Proposition \ref{prop:limitcaselindelof}, we may find a transitive node $W\in M\cap H(\beta)$ so that $s_2\cap M\seq W$. Then we may also add $W$ to $s_2$ to form a condition $s_3$ in the model sequence poset so that $W\in s_3$ and so that the nodes of $s_3$ above $W$ are precisely the nodes of $s_2$ from $M$ onward. We again have that the small nodes of $s_3$ are either small nodes of $s_2$ or are of the form $N\cap W$, where $N\in s_2$ is small. This implies that $\la s_3,f_2\ra$ is a condition.

    Now let $\dot{E}_{W, s_3\cap M}$ be the $(\ps\upharpoonright W)$-name consisting of all pairs $(\check{z},\la\bar{t},\bar{g}\ra)$ so that
    \begin{enumerate}
        \item $z\in X$ and $\la \bar{t},\bar{g}\ra\in\ps\upharpoonright W$;
        \item there is a condition $\la t,g\ra\in\ps\upharpoonright\beta$ so that
        \begin{enumerate}
            \item $s_3\cap M\seq t$,
            \item $\la t,g\ra\upharpoonright W=\la\bar{t},\bar{g}\ra$, and 
            \item $\la t,g\ra\Vdash_{\ps\upharpoonright\beta}\check{z}\in\dot{E}$.
        \end{enumerate}
    \end{enumerate}
    Note that this name is in $M$, being definable from parameters in $M$.\\

    \textbf{Claim 1:} $\la s_3,f_2\ra\upharpoonright W$ forces that $\check{x}\in\operatorname{cl}(\dot{E}_{W,s_3\cap M})$.\\

    \emph{Proof of Claim 1:} Since it suffices to work with basic open sets, and hence with members of $\tau$, fix $B\in\tau$ with $x\in B$, and fix an extension $\la\bar{t},\bar{g}\ra$ of $\la s_3,f_2\ra\upharpoonright W$. By Lemma \ref{lemma:transitivestronglygeneric}, we can find an extension $\la t',g'\ra$ of both $\la s_3,f_2\ra$ and $\la\bar{t},\bar{g}\ra$ and, by further extension if necessary, we may find a specific $z\in B$ with $\la t',g'\ra\Vdash\check{z}\in\dot{E}$. Then $\la t',g'\ra\upharpoonright W$ extends the starting condition $\la\bar{t},\bar{g}\ra$, and moreover, $(\check{z},\la t',g'\ra\upharpoonright W)$ is a member of $\dot{E}_{W,s_3\cap M}$, as witnessed by $\la t',g'\ra$. $\hfill\blacksquare\text{(Claim 1)}$\\

    \textbf{Claim 2:} $\la s_3,f_2\ra\upharpoonright W$ forces that $\check{x}\in\operatorname{cl}(M\cap\dot{E}_{W,s_3\cap M})$.\\

    \emph{Proof of Claim 2:} Again fix $B\in\tau$ with $x\in B$ and an extension $\la\bar{t},\bar{g}\ra$ of $\la s_3,f_2\ra\upharpoonright W$ in $\ps\upharpoonright W$. Fix a $V$-generic filter $G_W$ over $\ps\upharpoonright W$ containing $\la\bar{t},\bar{g}\ra$.

    Define $E_W:=\dot{E}_{W,s_3\cap M}[G_W]$ so that $x\in\operatorname{cl}(E_W)$ by the previous claim. Apply the elementarity of $M[G_W]$ to find a countable subset $E'_0$ of $E_W$ which is a member of $M[G_W]$ and which satisfies that $x\in\operatorname{cl}(E'_0)$. Fix a point $z\in E'_0\cap B$, noting that $z\in M$ (not just $M[G_W]$) by properness. Again by properness and elementarity, find a condition $\la\bar{s},\bar{f}\ra\in G_W\cap M$ so that $(\check{z},\la\bar{s},\bar{f}\ra)$ is a member of the name $\dot{E}_{W,s_3\cap M}$. Then $(\check{z},\la\bar{s},\bar{f}\ra)$ is a member of the name $M\cap\dot{E}_{W,s_3\cap M}$. Since $\la\bar{s},\bar{f}\ra\in G_W$, it is compatible with $\la\bar{t},\bar{g}\ra$, which suffices to finish the proof. $\hfill\blacksquare\text{(Claim 2)}$\\

    Recall that in order to argue that $\la s_1,f_1\ra\Vdash \check{x}\in\operatorname{cl}(M(\dot{E}))$, we are working to find an extension of $\la s_2,f_2\ra$ which forces that $\check{U}\cap\dot{E}\neq\es$. By Claim 2, we can find
    \begin{itemize}
        \item $\la\bar{s}_4,\bar{f}_4\ra$ extending $\la s_3,f_2\ra\upharpoonright W$ in $\ps\upharpoonright W$,
        \item a point $z\in M$ with $z\in U$,
        \item $\la\bar{s},\bar{f}\ra\in M\cap W$ with $\bar{s}\supseteq s_3\cap M$ and $\la\bar{s}_4,\bar{f}_4\ra$ extending $\la\bar{s},\bar{f}\ra$, and
        \item $\la s,f\ra\in M$ with $\la s,f\ra\upharpoonright W=\la\bar{s},\bar{f}\ra$ and which forces that $\check{z}\in\dot{E}$.
    \end{itemize}

    As in the proof of Proposition \ref{prop:limitcaselindelof}, let $\la s_4,f_4\ra$ be a condition which extends each of the conditions $\la\bar{s}_4,\bar{f}_4\ra$, $\la s_3,f_2\ra$, and $\la s,f\ra$. Then $\la s_4,f_4\ra$ forces that $z\in M(\dot{E})$ since it extends $\la s,f\ra$ and since $(\check{z},\la s,f\ra)\in M(\dot{E})$. Finally, $\la s_4,f_4\ra$ extends $\la s_3,f_2\ra$, and hence $\la s_2,f_2\ra$. Since $z\in U$, this completes the proof.
\end{proof}

\begin{proposition}\label{prop:weirdsuccessorcaseCT}
    Suppose that $\be\in Z^*\bsl\lim(Z^*)$ and that $\al:=\sup(Z^*\cap\be)\notin Z^*$. Suppose further that for all $\bar{\al}<\al$ with $\bar{\al}\in Z^*$, $\ps\upharpoonright\bar{\alpha}$ preserves that $X$ is countably tight. Then $\ps\upharpoonright\beta$ preserves that $X$ is countably tight. 
\end{proposition}
\begin{proof}
    Fix a condition $\la s_0,f_0\ra\in\ps\upharpoonright\beta$, a $(\ps\upharpoonright\beta)$-name $\dot{E}$ for a subset of $X$, and a specific point $x$ which is forced to be in the closure of $\dot{E}$. We will find an extension $\la s_1,f_1\ra$ of $\la s_0,f_0\ra$ and a name $\dot{E}_0$ for a countable subset of $\dot{E}$ so that
    $$
    \la s_1,f_1\ra\Vdash_{\ps\upharpoonright\beta}\check{x}\in\operatorname{cl}(\dot{E}_0).
    $$

    Let $M^*$ and $M$ be as in Propostion \ref{prop:weirdsuccessorcaseLindelof}, but containing the parameters relevant for this proposition. As usual, let $\la s_1,f_1\ra$ be an extension of $\la s_0,f_0\ra$ with $M\in s_0$, and define $M(\dot{E})$ to be the following name:
    $$
    M(\dot{E}):=\lb (\check{z},\la t,g\ra)\in M:z\in X\we\la t,g\ra\Vdash\check{z}\in\dot{E}\rb.
    $$
    We claim that
    $$
    \la s_1,f_1\ra\Vdash \check{x}\in\operatorname{cl}(M(\dot{E})).
    $$

    Thus we fix an extension $\la s_2,f_2\ra$ of $\la s_1,f_1\ra$ and an open $U\in\tau$ with $x\in U$, and we will find a further extension which forces that $\check{U}\cap M(\dot{E})\neq\es$.

    As in the proof of Proposition \ref{prop:weirdsuccessorcaseLindelof}, we extend $s_2$ in the model sequence poset to a condition $s_3$ so that the following hold:
    \begin{itemize}
        \item the top node of $s_3\cap H(\al)$, say $Q$, is a transitive node and is not in the domain of $f_2$,
        \item the nodes of $s_3\bsl s_2$ are either transitive or of the form $N\cap W$ for some $N\in s_2$ and $W\in\cal{T}$, and
        \item the nodes of $s_3$ above $Q$ are exactly the nodes of $s_2$ above $Q$.
    \end{itemize}
    For uniformity of notation, we set $f_3:=f_2$ so that $\la s_3,f_3\ra$ is a condition extending $\la s_2,f_2\ra$.

    Define $\dot{E}_{Q,s_3\cap M}$ to be the $(\ps\upharpoonright Q)$-name consisting of all pairs $(\check{z},\la\bar{t},\bar{g}\ra)$ so that
    \begin{enumerate}
        \item $z\in X$ is a point in the space,
        \item there exists $\la t,g\ra$ in $\ps\upharpoonright\beta$ so that
        \begin{enumerate}
            \item $\la t,g\ra\upharpoonright Q=\la\bar{t},\bar{g}\ra$,
            \item $t\supseteq s_3\cap M$, and
            \item $\la t,g\ra\Vdash\check{z}\in\dot{E}$.
        \end{enumerate}
    \end{enumerate}

\vspace{.1in}

    \textbf{Claim 1:} $\la s_3\cap Q,f_3\ra$ forces in $\ps\upharpoonright Q$ that $\check{x}\in\operatorname{cl}(\dot{E}_{Q,s_3\cap M})$.\\

    \emph{Proof of Claim 1:} Fix an extension $\la\bar{t},\bar{g}\ra$ of $\la s_3\cap Q,f_3\ra$ in $\ps\upharpoonright Q$ as well as an open $B\in\tau$ with $x\in B$. Since $\la\bar{t},\bar{g}\ra\leq\la s_3\cap Q,f_3\ra=\la s_3,f_3\ra\upharpoonright Q$ in $Q$, the conditions $\la\bar{t},\bar{g}\ra$ and $\la s_3,f_3\ra$ are compatible. Thus let $\la t',g'\ra$ extend both; we may assume by further extension if necessary that for a specific $z\in B$, $\la t',g'\ra\Vdash\check{z}\in\dot{E}$. Then $\la t',g'\ra\upharpoonright Q$ extends $\la s_3\cap Q,f_3\ra$, and $(\check{z},\la t',g'\ra\upharpoonright Q)$ is a member of the name $\dot{E}_{Q,s_3\cap M}$. \hfill$\blacksquare$(Claim 1)\\

    \textbf{Claim 2:} $\la s_3\cap Q,f_3\ra$ forces in $\ps\upharpoonright Q$ that $\check{x}\in\operatorname{cl}(M\cap\dot{E}_{Q,s_3\cap M})$.\\

    \emph{Proof of Claim 2:} Fix an extension $\la \bar{t},\bar{g}\ra$ of $\la s_3\cap Q,f_3\ra$ in $\ps\upharpoonright Q$ as well as an open $B\in\tau$ with $x\in B$. Let $G_Q$ be a $V$-generic filter over $\ps\upharpoonright Q$ containing $\la\bar{t},\bar{q}\ra$, and define $E_Q:=\dot{E}_{Q,s_3\cap M}[G_Q]$. Note that $E_Q\in M[G_Q]$. Apply the elementarity of $M[G_Q]$ to fix a countable subset $E'_0$ of $E_Q$ with $E'_0\in M[G_Q]$ and with $x\in\operatorname{cl}(E'_0)$. Fix a point $z\in B\cap E'_0$. Since $z\in M[G_Q]\cap E_Q$, we may find $\la\bar{s},\bar{f}\ra$ in $M[G_Q]$ so that $(\check{z},\la\bar{s},\bar{f}\ra)$ is in $\dot{E}_{Q,s_3\cap M}$. By properness, $(\check{z},\la\bar{s},\bar{f}\ra)$ is in fact in $M$. Since $\la\bar{s},\bar{f}\ra\in G_Q$, it is compatible with $\la\bar{t},\bar{g}\ra$, and this suffices to finish the proof of the claim. \hfill$\blacksquare$(Claim 2)\\

Now by applying Claim 2, we may find a pair $(\check{z},\la\bar{t},\bar{g}\ra)$ in $M\cap\dot{E}_{Q,s_3\cap M}$ with $\la s_3\cap Q,f_3\ra$ compatible with $\la\bar{t},\bar{g}\ra$. Let $\la t,g\ra$ witness that $(\check{z},\la\bar{t},\bar{g}\ra)$ is a member of $M\cap\dot{E}_{Q,s_3\cap M}$. Then, as in the proof of Proposition \ref{prop:weirdsuccessorcaseLindelof}, $\la s_3,f_3\ra$ and $\la t,g\ra$ are compatible conditions. This suffices to finish the proof of the proposition.
    
\end{proof}

\subsection{Putting it All Together}

We can now put together the previous pieces to complete the proof of the following theorem:

\begin{theorem}\label{theorem:CT}
    Under Assumption \ref{GroundModelassumption}, suppose that $X$ is a countably tight space. Then the poset $\ps_{CT}(X)$, where ``$CT$" denotes ``countably tight," forces that $\pfa_{CT}(X)$ holds.
\end{theorem}
\begin{proof}
    This proof is almost identical to the proof of Theorem \ref{theorem:Lindelof}, using the analogous preservation results for ``countably tight" in place of ``Lindel{\"o}f".

    In a bit more detail, we prove by induction on $\be\in Z^*\cup\lim(Z^*)$ that $\ps\upharpoonright\beta$ preserves that $X$ is countably tight.

    If $\be=\min(Z^*)$, then $\ps\upharpoonright\be$ is strongly proper, and hence by the results from \cite{GiltonHolshouser}, it preserves that $X$ is countably tight.

    Next, suppose that $\be\in\lim(Z^*)$ and that the result holds for all $\al$ below $\be$. By Proposition \ref{prop:limitcaseCT}, $\ps\upharpoonright\beta$ preserves that $X$ is countably tight.

    Next, consider the successor case where $\be\in Z^*\bsl\lim(Z^*)$. Define $\al:=\sup(Z^*\cap\be)$. If $\al\in Z^*$, then we apply Proposition \ref{prop:denseembedding} and Lemma \ref{lemma:SCAtight} to conclude that $\ps\upharpoonright\beta$ preserves that $X$ is countably tight. On the other hand, if $\al\notin Z^*$, then we apply Proposition \ref{prop:weirdsuccessorcaseLindelof} to conclude that $\ps\upharpoonright\beta$ preserves that $X$ is countably tight.

    Since the case $\beta=\theta$ was taken care of during our discussion of the limit case, we know that $\ps$ preserves that $X$ is countably tight. Just as in the proof of Theorem \ref{theorem:Lindelof}, we have that $\ps$ preserves $\om_1$, preserves all cardinals at least $\theta$, and forces that $\theta$ becomes $\aleph_2$. And finally, the usual arguments show that $\ps$ forces $\pfa_{CT}(X)$.
\end{proof}

\section{Comments on Applications and Questions}\label{Section:questions}

In follow-up work, we will discuss applications of these forcing axioms in more detail. For now, we make the following remark: any consequence $\Psi$ of $\pfa$ which can be forced by a strongly proper forcing is a consequence of the two $\pfa_\Phi(X)$ axioms discussed here. This is on account of the results from \cite{GiltonHolshouser} that strongly proper posets preserve the Lindel{\"o}f property and preserve countably tight spaces.

The following ``spectrum" question seems to be the most natural place to start:

\begin{question}
    Let $\Phi$ be either ``Lindel{\"o}f" or ``countably tight". What are the consequences of $\pfa_\Phi(X)$ as we vary the space $X$?
\end{question}

The following contains some specific instances of the previous question:

\begin{question}
    With $\Phi$ as in the previous question, for which $X$ does $\pfa_\Phi(X)$ imply the open graph axiom? The $P$-ideal dichotomy? The failure of square principles?
\end{question}

Another natural line of questions is the following:

\begin{question}
    For which other topological properties $\Phi$ and spaces $X$ satisfying $\Phi$ can we prove the consistency of $\pfa_\Phi(X)$? For instance, what if $X$ is a Rothberger space? A Menger space? Or if $X$ satisfies strategic versions of these covering principles?
\end{question}

\section{Acknowledgements}
I would like to thank Paul Gartside, Jared Holshouser, Pedro Marun, and Itay Neeman for a number of helpful conversations on this topic.


\begin{thebibliography}{0}

\bibitem{AsperoMota}
M. Asper{\'o} and M. A. Mota. Forcing consequences of PFA together with the continuum large. {\it Trans. Amer. Math Soc.} {\bf 367} (2015) no. 9 6103-6129.

\bibitem{DowSubmodels} 
A. Dow. An Introduction to Applications of Elementary Submodels to Topology. {\it Topology Proceedings} {\bf 11} (1988) 17-71.


\bibitem{DowTwoApplications}
A. Dow. Two Applications of Reflection and Forcing to Topology. {\it General Topology and Its Relations to Modern Analysis and Algebra VI, Proceedings Sixth Prague Topology Symposium 1986}. Helderman-Verlag, Berlin, 1988, pgs 155-172.

\bibitem{DowRemote}
A. Dow. Remote Points in Large Products. {\it Topology and its Applications} {\bf 16} (1983) 11-17.


\bibitem{DTWNewProofs}
A. Dow, F. D. Tall, and A. R. Weiss. New Proofs of The Consistency of the Normal Moore Conjecture I. {\it Topology and its Applications} {\bf 37} (1990) 33-51.


\bibitem{Friedman}
S. D. Friedman. Forcing with finite conditions. {\it Set Theory: Centre de Recerca Matematica Barcelona, 2003-2004}. Trends in Mathematics, Birk{\"a}user Verlag 2006, pgs 285-295.

\bibitem{GiltonHolshouser}
T. Gilton and J. Holshouser. Preservation of Topological Properties by Strongly Proper Forcings. Submitted (2024).

\bibitem{GiltonKruegerFun}
T. Gilton and J. Krueger. Mitchell's Theorem Revisited. {\it Ann. of Pure and Appl. Logic} {\bf 168} (2017) no. 5 922-1016.



\bibitem{GiltonNeeman}
T. Gilton and I. Neeman. Side Conditions and Iteration Theorems. From the 2016 Appalachian Set Theory conference. To appear in the next AST volume. Available at https://sites.pitt.edu/~tdg25/


\bibitem{GJTForcingAndNormality}
R. Grunberg, L.R. Junqueira, and F. D. Tall. Forcing and Normality. {\it Topology and its Applications} {\bf 84} (1998) 145-174.

\bibitem{Iwasa}
A. Iwasa. Covering Properties and Cohen Forcing. {\it Topology Proceedings} {\bf 31} (2007) 553-559.

\bibitem{Jech}
T. Jech. Set Theory. Springer Monographs in Mathematics. Springer-Verlag. 2003.


\bibitem{Kada}
M. Kada. Remarks on the Preservation of Topological Covering Properties Under Cohen Forcing.Preprint. https://arxiv.org/pdf/1002.4419.


\bibitem{Kada2}
M. Kada. Preserving the Lindel{\"o}f Property Under Forcing Extensions. {\it Topology Proceedings} {\bf 38} (2011) 237-251.


\bibitem{Koszmider}
P. Koszmider. On strong chains of uncountable functions. {\it Israel J. of Math.} {\bf 118} (2000) 289-315.


\bibitem{KruegerMota}
J. Kruger and M.A. Mota. Coherent Adequate Forcing and Preserving CH. {\it J. of Math. Logic} {\bf 15} (2015) no. 2.



\bibitem{KruegerForcingAxiom}
J. Kruger. A Forcing Axiom for a non-Special Aronszajn Tree. {\it Ann. of Pure and Appl. Logic} {\bf 171} (2020) no. 8 23 pp.

\bibitem{Marun}
P. Marun. Spaces, Trees, and Labelled Sets. PhD Thesis, Carnegie Mellon University, 2024.

 \bibitem{MitchellOuch}
 W. Mitchell. $I[\omega_2]$ can be the nonstationary ideal on cof$(\omega_1)$. {\it Trans. Amer. Math. Soc.} {\bf 361} (2009) no. 2 561-601.

\bibitem{Neemanttm}
I. Neeman. Forcing with sequences of models of two types. {\it Notre Dame J. Formal Logic} {\bf 55} (2014) 265-298.

\bibitem{ScheepersCountableTightness}
M. Scheepers. Remarks on Countable Tightness. {\it Topology and its Applications} {\bf 161} (2014) 407-432.


\bibitem{ShelahIndependence}
S. Shelah. Independence results. {\it J. Symbolic Logic} {\bf 45} (1980) no. 3 563-573.


\bibitem{ShelahPIPF}
S. Shelah. Proper and Improper Forcing. Perspectives in Mathematical Logic. Springer-Verlag, Berlin, 2nd edition, 1998.


\bibitem{SolovayTennenbaum}
R.M. Solovay and S. Tennenbaum. Iterated Cohen extensions and Souslin's problem. {\it Ann. of Math.} {\bf 94} (1971) no. 2 201-245.


\bibitem{TodorcevicSideConditions}
S. Todorcevi{\' c}. A note on the proper forcing axiom. Axiomatic set theory (Boulder, Colo., 1983) volume 31, Contemp. Math., Amer. Math. Soc., Providence, RI, 1984, 209-218.

\bibitem{TodorcevicDirected}
S. Todorcevi{\' c}. Directed sets and cofinal types. {\it Trans. Amer. Math. Soc.} {\bf 290} (1985) no. 2 711-723.

\bibitem{TodorcevicSouslinTree}
S. Todorcevi{\' c}. Forcing with a Coherent Souslin Tree. Preprint.

\end{thebibliography}
\end{document}